\documentclass[12pt]{amsart}

\usepackage{amssymb,amsmath,amsthm}
\usepackage{hyperref}
\usepackage[all]{xypic}
\usepackage{dsfont}

\entrymodifiers={+!!<0pt,\fontdimen22\textfont2>}
\newcommand{\Br}{\text{Br}}
\newcommand{\F}{\mathbb{F}}

\newcommand{\Gal}{\text{\rm Gal}}

\newcommand{\N}{\mathbb{N}}

\newcommand{\Z}{\mathbb{Z}}
\newcommand{\Q}{\mathbb{Q}}
\newcommand{\im}{\text{\rm{im}}}
\newcommand{\codim}{\text{\rm codim}}
\newcommand{\rk}{\text{\rm rk}}
\newcommand{\rkf}{\text{\rm f-rk}}
\newcommand{\rknf}{\text{\rm nf-rk}}
\newcommand{\ch}[1]{\textrm{char(}#1\textrm{)}}
\newcommand{\ve[1]}{\mathbf{\overrightarrow{#1}}}
\newcommand{\comment}[1]{}
\begin{document}

\title[Parameterizing Embedding Problems]{Parameterizing solutions to any Galois embedding problem over $\mathbb{Z}/p^n\mathbb{Z}$ with elementary $p$-abelian kernel}

\author[Andrew Schultz]{Andrew Schultz}
\address{106 Central Street, Wellesley College, Wellesley, MA 02482}
\email{andrew.c.schultz@gmail.com}

\begin{abstract}
In this paper we use the Galois module structure for the classical parameterizing spaces for elementary $p$-abelian extensions of a field $K$ to give necessary and sufficient conditions for the solvability of any embedding problem which is an extension of $\mathbb{Z}/p^n\mathbb{Z}$ with elementary $p$-abelian kernel.  This allows us to count the total number of solutions to a given embedding problem when the appropriate modules are finite, and leads to some nontrivial automatic realization and realization multiplicity results for Galois groups.
\end{abstract}

%\begin{keyword}
%Galois modules; Kummer theory; embedding problems; automatic realizations; realization multiplicity; group extensions
%\end{keyword}

\date{\today}

\maketitle

\newtheorem*{theorem*}{Theorem}
\newtheorem*{lemma*}{Lemma}
\newtheorem{proposition}{Proposition}[section]
\newtheorem{theorem}[proposition]{Theorem}
\newtheorem{corollary}[proposition]{Corollary}
\newtheorem{lemma}[proposition]{Lemma}

\theoremstyle{definition}
\newtheorem{definition}[proposition]{Definition}
\newtheorem*{definition*}{Definition}
\newtheorem*{remark*}{Remark}
\newtheorem{example}[proposition]{Example}

\parskip=10pt plus 2pt minus 2pt

\section{Introduction}\label{sec:introduction}

One of the fundamental problems in Galois theory is to determine conditions on a field $F$ which are necessary and sufficient for the appearance of a group $G$ as a Galois group over $F$; i.e., to determine when there exists an extension $K/F$ with $\Gal(K/F) \simeq G$.  The relative version of this question is the so-called embedding problem. For a given surjection of groups 
$\xymatrix{\hat G \ar@{->>}[r]^{\varphi} & G}$
and a given isomorphism $\psi_K:\Gal(K/F) \to G$, the embedding problem for $(\hat G,\varphi,\psi_K)$ over $K/F$ asks whether there is a field extension $L/F$ containing $K$ and an isomorphism $\psi_L:\Gal(L/F) \to \hat G$ such that the natural surjection from Galois theory makes the following diagram commute:
$$\xymatrix{\Gal(L/F)  \ar[d]_{\psi_L} \ar@{->>}[r] & \Gal(K/F) \ar[d]_{\psi_K}\\ \hat G \ar@{->>}[r]^{\varphi} & G}.$$  One can ask for a weaker solution to this embedding problem by only insisting that $\psi_L$ be an injection; in this case, $L$ is said to be a weak solution to the embedding problem.

There are a number of results in the literature which explore embedding problems for $p$-groups, particularly embedding problems whose kernel is $\Z/p\Z$: 
\begin{equation}\label{eq:embedding.problem.with.kernel.Fp}
\xymatrix{1 \ar[r] & \Z/p\Z \ar[r] & \hat G \ar[r]^\varphi & G \ar[r] & 1}.
\end{equation} These trace back to Dedekind's work on the embedding problem $\xymatrix{Q_8 \ar@{->>}[r]& \Z/2 \oplus \Z/2}$ in \cite{D}.  The interested reader can also find a bounty of results concerning the realizability of small $2$-groups as Galois groups (often by studying embedding problems) in articles such as \cite{DM,DP,GS1,GS2,GS3,GSS,M2,M4,M5}, as well as a number of papers on the realizability of small $p$-groups as Galois groups (again, often via embedding problems) in \cite{MN,M1,M3,Sw}.  

Away from characteristic $p$, the conventional method for approaching these problems is to assume $K$ contains the appropriate roots of unity and then consider the element $c \in H^2(G,\Z/p\Z)$ that corresponds to this extension of groups, with $\Z/p\Z$ identified with the trivial $G$-module $\mu_p$ of $p$th roots of unity in $K$.  The existence of an extension $L/F$ which solves the given embedding problem is then translated in terms of the image of this class $c$ within $\Br(K)$ under the map $H^2(G,\Z/p\Z) \to H^2(G,K^\times)$ which is induced by $\mu_p \hookrightarrow K^\times$; often this involves determining a specific algebra that represents this element within $\Br(K)$, and typically this is quite difficult.  When one doesn't have the necessary roots of unity, one approach is to solve the corresponding question in the extension of fields given by adjoining the necessary roots of unity, and then attempt to descend.  In characteristic $p$, one hopes to use the power of Witt's famous result from \cite{Wi} concerning the realizability of $p$-group as Galois groups in characteristic $p$; for instance, in \cite[App.~A]{JLY}, Jensen, Ledet and Yau use a technique similar to Witt's to show the embedding problem (\ref{eq:embedding.problem.with.kernel.Fp}) is solvable in characteristic $p$ provided it is central (i.e., $\ker(\varphi) \subseteq Z(\hat G)$) and nonsplit.

In this paper, we will show that there is a universal parameterizing module for solutions to embedding problems $\xymatrix{\hat G \ar@{->>}[r] & \Z/p^n\Z}$ whose kernel is an $\F_p[\Z/p^n\Z]$-module (i.e., an elementary $p$-abelian group with a $G$-action).  Though it has the same spirit as many of the embedding problems in the literature, we will develop our results without explicitly delving into $2$-cohomology.  Our parametrization involves studying the $\Gal(K/F)$-module structure of the parameterizing $\F_p$-space for elementary $p$-abelian extensions over $K$, which we denote $J(K)$; for instance, when $K$ contains a primitive $p$th root of unity, we will study the $\F_p[\Gal(K/F)]$-structure of $J(K) = K^\times/K^{\times p}$.  This study was initiated by Waterhouse in \cite{Wa}, and sections \ref{sec:group.classification} and \ref{sec:embedding.problems.in.terms.of.classical.parameterizing.spaces} from this paper can be thought of as a completion of the ideas that Waterhouse presents there.  

The question of studying embedding problems with elementary $p$-abelian kernel was also recently considered by Min\'{a}\v{c} and Swallow in \cite{MS2}.  In that paper, the authors consider embedding problems where the corresponding factor group $\Gal(K/F)$ is $\Z/p\Z$ and the kernel is a cyclic $\F_p[\Gal(K/F)]$-module.  Our paper generalizes these results by allowing $\Gal(K/F) \simeq \Z/p^n\Z$ for any $n \in \N$, and removes the condition of cyclicity (as a module) for the kernel.   Shirbisheh also considers non-cyclic kernels in \cite{Sh}, where he studies embedding problems over the field $\Q(\xi_{p^2})/\Q(\xi_p)$.  Aside from the fact that we have no restriction on the fields we consider, our approach differs in that we give explicit descriptions for all possible extensions of $\Z/p^n\Z$ by a finite $\F_p[G]$-module $A$, and then we find a parameterizing set for each such group within $J(K)$.  We are also more constructive in our approach to finding modules within $J(K)$ that solve a given embedding problem, giving a recipe for how one might build such a module first in terms of a fixed submodule and then through generators ``over" this fixed subspace.

To accomplish our goal, we will first classify all solutions to the ``group-theoretic embedding problem" $$\xymatrix{\hat G \ar@{->>}[r]^-{\varphi} & \Z/p^n\Z}$$ where the kernel $M$ is an elementary $p$-abelian group on which $\Z/p^n\Z$ acts.  We will show in Theorem \ref{th:counting.isom.types.of.group.embedding.problems} that any such $\hat G$ is determined by two pieces of data: the $\F_p[G]$-structure of $M$ and an integer $1 \leq \mu \leq p^n$ which can be thought of as a measure of how close the exact sequence $$\xymatrix{1 \ar[r] & M \ar[r] & \hat G \ar[r] & \Z/p^n\Z \ar[r] &1}$$ is to splitting.  For example, the group $M \rtimes \Z/p^\Z$ corresponds to $\mu = p^n$.  In general, we will adopt the notation $M \bullet_\mu \Z/p^n\Z$ to denote the group determined by a given $M$ and $\mu$.  (Section \ref{sec:group.classification} describes this taxonomy in greater detail.)

We now state the main result of the paper, which shows that $J(K)$ acts as the universal parameterizing space for embedding problems expressible as an extension of $\Z/p^n\Z$ by an elementary $p$-abelian group.  The statement of the result uses two maps, $N: J(K) \to J(F)$ and $\iota: J(F) \to J(K)$; the first is induced by either an appropriate norm or trace map (depending on the presence of roots of unity), and the second is induced by the inclusion of fields.  We will have more to say about these maps in section \ref{sec:embedding.problems.in.terms.of.classical.parameterizing.spaces}.

\begin{theorem}\label{th:main.theorem}
Let $\Gal(K/F) \simeq \Z/p^n\Z$, and let $M \bullet_\mu \Z/p^n\Z$ be a given extension of $\Z/p^n\Z$ by $M$.  Then the extensions $L/F$ which solve the embedding problem $\xymatrix{M \bullet_\mu \Z/p^n\Z \ar@{->>}[r] & \Z/p^n\Z}$ over $K/F$ are parameterized by $\F_p[\Gal(K/F)]$-submodules  $U \subseteq J(K)$ which satisfy $M \simeq U$ and  $$\inf_{u \in U}\left\{\dim_{\F_p} \langle u\rangle_{\F_p[G]}: u \in \ker(\iota \circ N) \setminus \ker(N)\right\} = \mu.$$
\end{theorem}

(Of course, it could be the case that $U \cap \left(\ker(\iota \circ N)\setminus  \ker(N)\right) = \emptyset$; since cyclic $\F_p[G]$-modules have dimension between $0$ and $p^n$, the appropriate interpretation for $\inf(\emptyset)$ in this case is $\inf(\emptyset) = p^n$.  Since we've already mentioned that $M \bullet_{p^n} G$ corresponds to $M \rtimes G$, this statement simply means that solutions to $M \rtimes G$ correspond to those modules $U$ with $M \simeq U$ and $U \cap \ker(\iota \circ N) = U \cap \ker(N)$.)

When the module structure of $J(K)$ for a given field $K$ is known, this result tells us we can use linear algebra to answer all questions about embedding problems over $K/F$ whose kernel is an elementary $p$-abelian group: precisely which embedding problems have solutions, together with an explicit count on those extensions.  Fortunately, the module structure of $J(K)$ has already been calculated, and can in fact be expressed in terms of some basic arithmetic invariants of $K/F$ (we will review this work in section \ref{sec:general.applications}).  Hence we can use the above result to deduce some very concrete realization and enumeration results.  We finish this section by detailing a few of these.

To state our results we will also need to establish some notation for modules over $\F_p[\Z/p^n\Z]$; this will be reviewed more fully in section \ref{sec:notation}.  When $\Z/p^n\Z = \langle \sigma \rangle$, we will write $\Psi = \sigma-1$.  For an $\F_p[\Z/p^n\Z]$-module $M$, the subspaces $M_{\{\ell\}} := \im\left(\xymatrix{M \ar[r]^{\Psi^{\ell-1}}& M}\right) \cap M^G$ provide a filtration that we call the length filtration on $M$, denoted $F^{\mbox{\tiny{len}}}_M$: $$F^{\mbox{\tiny{len}}}_M: \quad M^G =M_{\{1\}} \supseteq M_{\{2\}} \supseteq \cdots \supseteq M_{\{p^n\}} \supseteq M_{\{p^n+1\}} = \{0\}.$$  This filtration captures the module structure of $M$ itself; writing $d_\ell = \codim\left(M_{\{\ell+1\}},M_{\{\ell\}}\right)$,  one can show \begin{equation}\label{eq:decomposition.for.general.module}
M \simeq \bigoplus_{\ell=1}^{p^n} \oplus_{d_\ell} \F_p[\Z/p^n\Z]/\langle \Psi^{\ell}\rangle.
\end{equation}

Now we recall a definition from \cite{MSS1}.  For a field extension $K/F$ with $\Gal(K/F) \simeq \Z/p^n\Z$, let $K_i$ denote the intermediate field of degree $p^i$ over $F$.  If the embedding problem
$$\xymatrix{\Z/p^{n+1}\Z \ar@{->>}[r] & \Z/p^n\Z }$$ for $K/F$ has a solution, then define $i(K/F) = -\infty$.  Otherwise, let $s$ be the minimum value such that the embedding problem
$$\xymatrix{\Z/p^{n-s+1}\Z \ar@{->>}[r] & \Z/p^{n-s}\Z}$$ for $K/K_s$ has a solution, and define $i(K/F) = s-1$.  Notice that we have $i(K/F) \in \{-\infty,0,\cdots,n-1\}$ provided $K \neq F$.  

It was shown in \cite{MSS1} that $i(K/F)$ is one of the defining characteristics of the $\F_p[\Gal(K/F)]$-module structure of the parameterizing space of elementary abelian extensions of $K$. In the Kummer case, the other necessary data to determine the $\F_p[\Gal(K/F)]$-structure of this space is the size of the various norm subgroups within $F$ (we will review this --- together with the necessary adjustments outside the Kummer case --- in section \ref{sec:general.applications}).  For our results below, we encode this data in a seemingly peculiar way, though one that is well-suited for making statements about embedding problems.  For $1 \leq \ell \leq p^n$, define $\varepsilon(\ell)$ by $p^{\varepsilon(\ell)-1} <i \leq p^{\varepsilon(\ell)}$, and if $\xi_p \in K$ let
$$\mathfrak{D}_{\ell} :=\dim_{\F_p}\left(\frac{\left(N_{K_{\varepsilon(\ell)}/F}(K^{\times}_{\varepsilon(\ell)})\right)K^{\times p}}{K^{\times p}}\right).$$
%(We will see later that these quantities are related to the length filtration for $K^\times/K^{\times p}$.)

Our first theorem gives specific conditions on the solvability of the embedding problem $\xymatrix{M \rtimes \Z/p^n\Z \ar@{->>}[r] & \Z/p^n\Z}$ over $K/F$.  For convenience we state this result only in the case $\xi_p \in K$, though we'll see later in section \ref{sec:counting} that a few notational changes will allow it to apply to all fields.  The result uses the notation $\binom{n}{m}_p$ for the $p$-binomial coefficient which we define in section \ref{sec:counting}.  %Finally, for a real number $\alpha$ we write $\lceil \alpha \rceil$ for smallest integer $n$ satisfying $\alpha \leq n$ (i.e., $\lceil \cdot \rceil$ is the familiar ceiling function); likewise $\lfloor \alpha \rfloor$ is the largest integer $n$ satisfying $n \leq \alpha$.

\begin{theorem}\label{th:counting.to.split.embedding.problem}
Suppose that $\Gal(K/F) \simeq \Z/p^n\Z$ and $\xi_p \in K$.  Suppose further that $M$ is a finite $\F_p[\Gal(K/F)]$-module.    Then the embedding problem $\xymatrix{M \rtimes G \ar@{->>}[r] &G}$ has a solution over $K/F$ if and only if $\dim\left(M_{\{\ell\}}\right) \leq \mathfrak{D}_{\ell}$ for all $1 \leq \ell \leq p^n$. %Then the embedding problem $\xymatrix{A \rtimes G \ar@{->>}[r] & G}$ has a solution over $K/F$ if and only if $\Delta(A_{\{i\}}) \leq \mathfrak{D}_{\{i\}}$ for all $1 \leq i \leq p^n$.  

If $(F^\times K^{\times p})/K^{\times p}$ is infinite and the embedding problem $\xymatrix{M \rtimes G \ar@{->>}[r] & G}$ is solvable, then there are infinitely many solutions to this embedding problem over $K/F$.  If $(F^\times K^{\times p})/K^{\times p}$ is finite and the embedding problem $\xymatrix{M \rtimes G \ar@{->>}[r] & G}$ is solvable, then the number of solutions to this embedding problem over $K/F$ is
$$p^{\dim\left(M_{\{p^n\}}\right)}~\prod_{\ell=1}^{p^n}\binom{\mathfrak{D}_{\ell} - \dim(M_{\{\ell+1\}})}{\codim(M_{\{\ell+1\}},M_{\{\ell\}})}_p~\left(p^{\sum_{j<\ell} \mathfrak{D}_{j} - \dim(M_{\{j\}})}\right)^{\codim\left(M_{\{\ell+1\}},M_{\{\ell\}}\right)}$$ unless $p=2,n=1$ and $i(K/F)=0$.  In this latter case, the number of solutions is instead
$$\prod_{\ell=1}^{p^n}\binom{\mathfrak{D}_{\ell} - \dim(M_{\{\ell+1\}})}{\codim(M_{\{\ell+1\}},M_{\{\ell\}})}_p~\left(p^{\sum_{j<\ell} \mathfrak{D}_{j} - \dim(M_{\{j\}})}\right)^{\codim\left(M_{\{\ell+1\}},M_{\{\ell\}}\right)}. $$
\end{theorem}

%The semi-direct product is one of the most natural ways to construct a group out of the $G$-module $A$, and hence this result gAs an example of a group which is of the form $A \rtimes G$, we observe that the nonabelian group of order $p^3$ and exponent $p$ --- the so-called Heisenberg group $H_{p^3}$ --- is the group $A_2 \rtimes \Z/p$.  

%Though this result is expressed only for fields containing a primitive $p$th root of unity, we'll see later that this result holds for any extension of fields $K/F$ with $\Gal(K/F) \simeq \Z/p^n\Z$ (after an appropriate translation of the constants $\mathfrak{D}_{\{i\}}$ and the space $(F^\times K^{\times p})/K^{\times p}$). 

Our parameterization also allows us to make a number of statements about how the appearance of one group as a Galois group over $F$ influences the existence of other groups as Galois groups over $F$.  To preview some results of this flavor, we introduce the following definition. For a given (finite) group $G$ and field $E$, we say that an extension $L/E$ is a $G$-extension of $E$ if $\Gal(L/E) \simeq G$; we will write $\mathfrak{F}(G)$ for the set of all fields $E$ which admit a $G$-extension.  

If $E \in \mathfrak{F}(G)$ implies $E \in \mathfrak{F}(Q)$, then we say that $G$ automatically realizes $Q$; the automatic realization result is said to be trivial when $Q$ is a quotient of $G$, since in this case $E \in \mathfrak{F}(Q)$ by elementary Galois theory. Classic examples of nontrivial automatic realizations were given by Whaples in \cite{Wh}, where he showed that $\Z/p\Z$ automatically realizes $\Z/p^n\Z$ when $p$ is an odd prime and $n\geq 2$, as well as showing $\Z/4\Z$ automatically realizes $\Z/2^n\Z$ for all $n \geq 3$.   Jensen has written a number of excellent articles on automatic realizations, including \cite{J1,J2,J3}, and there are other automatic realizations considered in \cite{Br,GS1,L,M1,Wh}.  It is worth noting that prior to \cite{MSSauto}, the known automatic realization results for nonabelian $p$-groups with $p>2$ were extremely limited, and seem to not involve groups of order larger than $p^4$; the module-theoretic machinery used in \cite{MSSauto} and this paper, on the other hand, provide several infinite classes of automatic realizations for non-abelian $p$-groups, and even many cases where one (relatively) small $p$-group automatically realizes a (relatively) larger $p$-group.

We now give an automatic realization that generalizes those found in \cite{MSSauto}; this result is a special case of the results we develop in section \ref{sec:general.applications}, and in particular we will have similar results for nonsplit groups as well.  To state the result, recall that for $d_\ell = \codim\left(M_{\{\ell+1\}},M_{\{\ell\}}\right)$ we have $M \simeq \bigoplus_{\ell=1}^{p^n} \oplus_{d_\ell} \F_p[\Z/p^n\Z]/\langle \Psi^{\ell-1}\rangle$.  We define $\lceil M \rceil$ to be the $\F_p[\Z/p^n\Z]$-module
\begin{equation}\label{eq:roundup.decomposition.for.general.module}\lceil M \rceil := \bigoplus_{t=1}^{n} \oplus_{D_t} \F_p[\Z/p^n\Z]/\langle \Psi^{p^{t}} \rangle,\end{equation} where $D_t := \codim\left(M_{\{p^t+1\}},M_{\{p^{t-1}+1\}}\right) = d_{p^{t-1}+1}+\cdots+d_{p^t}$. (We have chosen the notation $\lceil M \rceil$ suggestively, since we think of it as taking each summand of $M$ of length $p^{t-1}+j$ with $1 \leq j \leq p^t-p^{t-1}$ and ``rounding it up" to a module of dimension $p^{t}$.) Note, that, in general, $\lceil M \rceil$ is much larger than $M$, and hence $\lceil M \rceil \rtimes \Z/p^n\Z$ is typically much larger than $M \rtimes \Z/p^n\Z$.  For example, if $d_2=k$ and all other $d_i=0$, then $$\left|\left\lceil M \right\rceil\rtimes \Z/p^n\Z\right|=p^{(p-2)k}\left|M\rtimes \Z/p^n\Z\right|.$$

\begin{theorem}\label{th:roundup.automatic.realization}
If $F \in \mathfrak{F}(M \rtimes \Z/p^n\Z)$, then $F \in \mathfrak{F}(\lceil M \rceil\rtimes \Z/p^n\Z).$
\end{theorem}

%\begin{example}
%When $n=1$ and $d_2=k$ (with all other $d_i=0$), this theorem tells us that there is a nonabelian group of order $p^{1+2k}$ which automatically realizes a nonabelian group of order $p^{1+pk}$.
%\end{example}

$M \rtimes \Z/p^n\Z$ is naturally a quotient of $\lceil M \rceil \rtimes \Z/p^n\Z$ since $\ell \leq p^{\varepsilon(\ell)}$, and so  $\lceil M \rceil \rtimes \Z/p^n\Z$ trivially realizes $A \rtimes \Z/p^n\Z$; this theorem says that the opposite (and highly nontrivial) automatic realization also holds.    %Theorem \ref{th:roundup.automatic.realization} is the natural generalization of the main result from \cite{MSSauto}; more general automatic realization results for nonsplit groups will also be discussed in Theorem \ref{th:general.auto.realization}.  

%To put this last theorem in context, suppose that $A$ is a direct sum of two cyclic submodules, each with dimension $2$.  In this case $\lceil A \rceil$ is a direct sum of two cyclic submodules, each with dimension $p$.  The group $A \rtimes G$ has order $p^{n+4}$, whereas the group $\lceil A \rceil \rtimes G$ has order $p^{n+2p}.

To take advantage of the fact that we have precise counts on the number of solutions to a given embedding problem, we also state some results concerning realization multiplicity.  Let $\nu(G,F)$ denote the number of distinct $G$-extensions of $E$ within a fixed algebraic closure of $E$, and $$\nu(G) = \min_{F \in \mathfrak{F}(G)} \nu(G,E).$$  This latter quantity is called the realization multiplicity of $G$.  Jensen has explored realization multiplicities in \cite{JP1,JP2}.   We have a generalization of the main result from \cite{BS}.

\begin{theorem}\label{th:free.summands.realization.multiplicity}
Suppose that $M$ is an $\F_p[\Z/p^n\Z]$-module satisfying either 
\begin{itemize}
\item $\codim\left(M_{\{\ell+1\}},M_{\{\ell\}}\right) >1$ for some $\ell=p^i+j$ where $2 \leq j < p^{i+1}-p^i$; or 
\item for $p>2$, $\sum_{i=0}^{n-1} \codim\left(M_{\{p^i+2\}},M_{\{p^i+1\}}\right) >1$; or
\item for $p=2$, $\sum_{i=1}^{n-1} \codim\left(M_{\{p^i+2\}},M_{\{p^i+1\}}\right) >1$.
\end{itemize}  Then for any $\hat G$ that is any extension of $\Z/p^n\Z$ by $M$, we have $\nu(\hat G) \geq p^{\dim\left(M_{\{p^n\}}\right)}$.
%and if there exists $m_1,\cdots,m_k \in M$ so that $\langle m_1,\cdots,m_k\rangle \simeq \oplus_k \F_p[\Z/p^n\Z]$, then $\nu(\hat G) \geq p^k$.
\end{theorem}

These automatic and realization multiplicity results provide some \emph{a priori} unexpected restrictions on the structure of absolute Galois groups, and give some tangible characteristics that distinguish them from the larger class of profinite groups.

This paper proceeds as follows.  In the next section we remind the reader of some results about $\F_p[G]$-modules, and we classify all extensions of $\Z/p^n\Z$ by $\F_p[\Z/p^n\Z]$-modules in section \ref{sec:group.classification}.  In section \ref{sec:embedding.problems.in.terms.of.classical.parameterizing.spaces} we consider the parameterizing space of elementary $p$-abelian extensions over $K$ --- denoted $J(K)$ --- and some of the known bijections between cyclic submodules of $J(K)$ and extensions of $\Z/p^n\Z$; we extend these results to include the case of characteristic $p$, and we then describe the collection of submodules in $J(K)$ that correspond to fields that  are solutions to embedding problems $\xymatrix{\hat G \ar@{->>}[r] & \Z/p^n\Z}$ over $K/F$. This allows us to give a precise count for the number of such solutions, which we do in section \ref{sec:counting}.  In section \ref{sec:general.applications} we recall some of the known results about the module structure of $J(K)$ when $\ch{F} \neq p$, and we extend these results to include $\ch{F} = p$ as well.  We then use these to make statements about realization multiplicities and automatic realizations.

\section{Notation and $\F_p[G]$-decompositions}\label{sec:notation}

Throughout the paper $p$ will denote a prime number, and we will use $K/F$ to denote an extension such that $G=\Gal(K/F) = \langle \sigma \rangle \simeq \Z/p^n\Z$ with $n \in \Z^+$.  We use $\mathds{1}_S$ as the indicator function for a subset $S$ of the natural numbers; often we describe $S$ explicitly in terms of equalities or inequalities.  

Our results will concern embedding problems $\xymatrix{\hat G \ar@{->>}[r]^{\varphi}& G}$ over $K/F$, where $\hat G$ is an extension of $G$ by a finite $\F_p[G]$-module $M$.  We will suppress the explicit isomorphism $\psi_K:\langle \sigma \rangle \to \Z/p^n\Z$ when considering these embedding problems.  To emphasize the Galois-theoretic motivation of our work, we will call the extensions of $G$ by $M$ ``group-theoretic embedding problems."  We say that two group-theoretic embedding problems $(\hat G_1,\varphi_1)$ and $(\hat G_2,\varphi_2)$ are isomorphic if there exists an isomorphism of groups $\psi:\hat G_1 \to \hat G_2$ that makes the following diagram commute:
$$\xymatrix{
\hat G_1 \ar@{->>}[r]^{\varphi_1} \ar[d]_{\psi} & G \ar@{=}[d]\\
\hat G_2 \ar@{->>}[r]^{\varphi_2}  & G
.}$$  If we wish to assemble all embedding problems over $G$ into a reasonable category, the morphisms of interest will be surjections: we'll be searching for solutions that come from Galois theory, and the only interesting morphisms of fields are injections.

When we work with an $\F_p[G]$-module, we will assume that the underlying vector space structure is written additively, and hence the $G$-action will be written additively.  

We now collect certain key facts about finite $\F_p[G]$-modules.  A more detailed exposition can be found in \cite[Sec.~1.1]{MSS1} or \cite[Sec.~2.3]{LMSSembed}.  Of central importance to the ring $\F_p[G]$ is the element $\sigma-1$; we will use this element enough that it is convenient to use the abbreviation $$\Psi := \sigma-1.$$  The ideals in $\F_p[G]$ are simply $\{\langle\Psi^\ell\rangle: 1 \leq \ell \leq p^n-1\}$, and hence any cyclic submodule with $\F_p$-dimension $\ell$ is isomorphic to $\F_p[G]/\langle \Psi^\ell\rangle$.  One can show that these are the only indecomposable $\F_p[G]$-modules, and moreover that for any $\F_p[G]$-submodule $M$ there are non-negative integers $\{d_i\}_{i=1}^{p^n}$ so that 
\begin{equation}\label{eq:general.module.decomposition}
M \simeq \bigoplus_{\ell=1}^{p^n} \oplus_{d_\ell} \F_p[G]/\langle \Psi^\ell \rangle.
\end{equation}  This decomposition (and hence the collection $\{d_\ell\}_{\ell=1}^{p^n}$) is unique up to permutation of the summands.  For a given element $m\in M$, we will call $\dim_{\F_p}\langle m \rangle$ the length of $m$, which we write as $\ell(m)$.

It will occasionally be helpful to know the number of various generators of an $\F_p[G]$-module $M$.  Following the notation from the decomposition (\ref{eq:general.module.decomposition}), we write
\begin{equation*}\begin{split}
\rk(M) = \sum_{\ell=1}^{p^n}d_{\ell},\quad \rkf(M) = d_{p^n} \quad \mbox{ and } \quad \rknf(M) = \sum_{\ell=1}^{p^n-1} d_\ell.
\end{split}\end{equation*}  We will call these quantities the rank, free rank and non-free rank, respectively.

The following proposition gives us a way to build an $\F_p[G]$-module from its fixed submodule.  

\begin{proposition}\label{prop:build.a.module}
Suppose that $M$ is a finite $\F_p[G]$-module and consider the filtration of $\F_p$-subspaces $M^{G} = M_{\{1\}} \supseteq M_{\{2\}} \supseteq \cdots \supseteq M_{\{p^n\}} \supseteq M_{\{p^n+1\}} = \{1\}$, where \begin{equation}\label{eq:subbrace.notation}M_{\{\ell\}} = \im\left(\xymatrix{M \ar[r]^-{\Psi^{\ell-1}}& M}\right) \cap M^{G}.\end{equation} Let $\mathcal{I}_{\ell}$ be chosen so that $\cup_{i \geq \ell} \mathcal{I}_i$ is a basis for $M_{\{\ell\}}$, and for each $x \in \mathcal{I}_\ell$ let $\alpha_x \in M$ be given so that $\Psi^{\ell-1}\alpha_x = x$.  Then $M = \bigoplus_{\ell=1}^{p^n} \oplus_{x \in \mathcal{I}_\ell} \langle \alpha_x \rangle$.
\end{proposition}

\begin{proof}
Each of the submodules are independent by \cite[Lm.~2]{MSS1}, and so the stated sum is direct.  The containment ``$\supseteq$" is obvious.  For the opposite containment, we prove that each $m \in M$ is contained in $\tilde M := \oplus_{\ell=1}^{p^n} \oplus_{x \in \mathcal{I}_\ell} \langle \alpha_x \rangle$ by induction on the length of $m$.  If $\ell(m) = 1$ then $m \in \tilde M$ since $M_{\{1\}} = M^G$.  Now suppose $\tilde M$ contains all elements of length at most $\ell-1$, and suppose $\ell(m) = \ell$.  Then $\Psi^{\ell-1}m \in M_{\{\ell\}}$, and hence there exists constants $c_x \in \F_p$ such that
$$\Psi^{\ell-1}m = \mathop{\sum_{i\geq \ell}}_{x \in \mathcal{I}_i} c_x x = \mathop{\sum_{i\geq \ell}}_{x \in \mathcal{I}_i} c_x \Psi^{i-1} \alpha_x =  \Psi^{\ell-1} \mathop{\sum_{i\geq \ell}}_{x \in \mathcal{I}_i} c_x \Psi^{i-\ell }\alpha_x .$$  Hence the element $ \left(\sum c_x \Psi^{i-\ell} \alpha_x\right)- m$ has length less than $\ell$, and is therefore contained in $\tilde M$.  Since each of the $\alpha_x \in \tilde M$ as well, this forces $\alpha \in \tilde M$, as desired.
\end{proof}

\begin{corollary}\label{cor:summand.multiplicities.and.codimension.in.length.filtration}
Suppose $M$ is a finite $\F_p[G]$-module.  Then $M \simeq \oplus_{\ell=1}^{p^n} \oplus_{d_\ell} \F_p[G]/\langle \Psi^\ell\rangle$ if and only if for all $1 \leq \ell \leq p^n$ we have $\codim\left(M_{\{\ell+1\}},M_{\{\ell\}}\right) = d_\ell.$
\end{corollary}

\section{Classifying groups}\label{sec:group.classification}

We are interested in classifying extensions $\xymatrix{\hat G \ar@{->>}[r] & G}$ for which the kernel is elementary $p$-abelian (recall that $G = \Gal(K/F) = \langle \sigma \rangle \simeq \Z/p^n\Z$ and $\Psi:=\sigma-1$).  In order to be slightly more precise, start with the data of the group $G$ (whose operation is written multiplicatively) and a $G$-module $M$ which is an  elementary $p$-abelian as a group (whose operation is written additively).  We say that $\xymatrix{\hat G \ar@{->>}[r] & G}$ is an embedding problem with kernel $M$ if in the short exact sequence
$$\xymatrix{1 \ar[r] &M \ar[r]^\iota &\hat G \ar[r]^\varphi &G \ar[r] & 1}$$ satisfies the condition that the action of $G$ on $M$ by conjugation is compatible with the $G$-action on $M$: for every $\tau \in G$ and $m \in M$, and for any $\hat \tau \in \hat G$ satisfying $\varphi(\hat \tau) = \tau$, we have $$\hat \tau^{-1}\iota(m)\hat \tau = \iota(\tau \cdot m).$$ (To emphasize the $G$-action on $M$ in this section --- where there are several operations at play --- we use the usual ``$\cdot$" notation.  We'll also adopt the standard abuse of notation of dropping the $\iota$ when we consider elements of $M$ as elements of $G$.) The following definition gives one natural way to generate groups $\hat G$ as above.%Throughout the balance of the paper, we will be interested in studying the extensions of $G \simeq \Z/p^n\Z$ by a given $\F_p[G]$-module $M$.  

\begin{definition}\label{def:initial.group.extension}
Let $M = \oplus_{i=1}^{\rk(M)} \langle \alpha_i \rangle$ be a finite $\F_p[G]$-module, and let $\ve[c] \in \F_p^{\rk(M)}$ be given.  We define $\mathfrak{G}(M,\ve[c])$ to be the group generated by $\{\alpha_i\}_{i=1}^{\rk(M)} \cup \{\hat \sigma\}$ and subject to the relations
\begin{enumerate}
\item\label{it:group.ext.one} $\alpha_i + \alpha_j = \alpha_j + \alpha_i,$
\item\label{it:group.ext.two} $\hat \sigma \alpha_i \hat \sigma^{-1} = \sigma \cdot \alpha_i, $
\item\label{it:group.ext.three} $\Psi^{\ell(\alpha_i)}\cdot \alpha_i=0,$ and
\item\label{it:group.ext.four} $\hat \sigma^{p^n} = \sum_{i=1}^{\rk(A)} c_i \Psi^{\ell(\alpha_i)-1} \cdot \alpha_i$.
\end{enumerate} 
In particular, $\mathfrak{G}(M,\ve[c]) = \{m \hat \sigma^j: m \in M, 0\leq j <p^n\}$ as a set, and the group operation is given by
$$(m_1 \hat \sigma^j)(m_2 \hat \sigma^k) = \left\{\begin{array}{ll} 
(m_1 + \sigma^j\cdot m_2)\hat \sigma^{j+k},&\mbox{ if }j+k<p^n\\
(m_1 + \sigma^j\cdot m_2 + \sum_{i=1}^{\rk(M)} c_i \Psi^{\ell(\alpha_i)-1}\cdot \alpha_i)\hat \sigma^{j+k-p^n},&\mbox{ if }j+k\geq p^n.
\end{array}\right.
$$  The group-theoretic embedding problem for $\mathfrak{G}(M,\ve[c])$ over $G$ is then $\xymatrix{\mathfrak{G}(M,\ve[c]) \ar@{->>}[r]^-{\varphi} & G}$, where $\varphi$ is defined by $\varphi(\alpha_i) = 1$ and $\varphi(\hat \sigma) = \sigma$; we will often abuse notation and speak of the embedding problem $\mathfrak{G}(M,\ve[c])$ without referring to either $G$ or $\varphi$.
\end{definition}

\begin{example}\label{ex:hp3.and.mp3}
When $p$ is odd, there are two nonabelian groups of order $p^3$. One has exponent $p$ (which we will denote $H_{p^3}$) and the other has exponent $p^2$ (which we will call $M_{p^3}$).  If we let $G = \langle \sigma \rangle \simeq \Z/p\Z$ and $M = \langle \alpha_1 \rangle \simeq \F_p[G]/\langle\Psi^2\rangle$, then it is not hard to see that $\mathfrak{G}(M,\ve[0])$ and $\mathfrak{G}(M,\ve[e]_1)$ are nonabelian groups of order $p^3$ with $H_{p^3} \simeq \mathfrak{G}(M,\ve[0])$ and $M_{p^3} \simeq \mathfrak{G}(M,\ve[e]_1)$.  Indeed, any nonzero $\ve[v] \in \F_p^1$ has $M_{p^3} \simeq \mathfrak{G}(M,\ve[v])$.
\end{example}

\begin{proposition}\label{prop:every.extension.is.of.general.form}
If $\xymatrix{\hat G \ar@{->>}[r] & G}$ is a group-theoretic embedding problem with kernel given by a finite $\F_p[G]$-module $M$, where $M = \oplus_{i=1}^{\rk(M)} \langle \alpha_i \rangle$, then there exists some $\ve[c] \in \F_p^{\rk(M)}$ such that $\xymatrix{\hat G \ar@{->>}[r] & G}$ is isomorphic to the embedding problem $\mathfrak{G}(M,\ve[c])$.
\end{proposition}

\begin{proof}
We know that $\hat G$ is generated by $\F_p[G]$-generators $\{\alpha_i\}_{i=1}^{\rk(M)}$ for $M$ together with a lift $\hat \sigma \in \hat G$ of $\sigma \in G$. Clearly the relations satisfied by $M$ appear in the relations for such an extension of groups, which verifies (\ref{it:group.ext.one}-\ref{it:group.ext.three}).
%hence if $A = \oplus_{i=1}^{\rk(A)} \langle \alpha_i \rangle$ then we have 
%\newcounter{saveenum}
%\begin{enumerate}
%\item $\alpha_i \alpha_j = \alpha_j \alpha_i$ for $1 \leq i,j \leq \rk(A)$;
%\item $\hat \sigma \alpha_i \hat \sigma^{-1} = \alpha_i^\sigma$ for $1 \leq i \leq \rk(A)$; and
%\item $\alpha_i^{(\sigma-1)^{\ell(\alpha_i)}}=1$ for $1 \leq i \leq \rk(A)$.
%\setcounter{saveenum}{\value{enumi}}
%\end{enumerate}
The last data that determines the structure of such an extension is the value of $\hat \sigma^{p^n}$.  This element must lie in $M$ since it has trivial image in $G$, and it must be fixed by the action of $\sigma$ as well.  Since the fixed submodule of $M$ is generated by $\left\{\Psi^{\ell(\alpha_i)-1}\cdot \alpha_i\right\}_{i=1}^{\rk(M)}$, this means that for some $\ve[c] \in \F_p$, we have
$$\hat \sigma^{p^n} = \sum_{i=1}^{\rk(M)} c_i \Psi^{\ell(\alpha_i)-1}\cdot \alpha_i.$$
\end{proof}
%\begin{definition*}
%The group $\mathfrak{G}(A_\ell,c)$ is the group generated by $\alpha$ and $\hat \sigma$ subject to relations (1)-(3) above.
%\end{definition*}

%We now give a natural first step towards a generalization of Waterhouse's classification.  To state this result, it will be convenient to write an $\F_p[G]$-decomposition of $A$ for which generators are explicitly named.

%The previous discussion provides the justification for the following

%\begin{proof}
%As in the case of a cyclic module, the first two relations are imposed on us by insisting that $\hat G$ be an extension of $G$ by $A$ and by the module structure of $A$ itself.  Likewise, since the elements $\{w_1,\cdots,w_s\}$ each commute in the $\F_p[G]$-module $A$, the must also compute in $\hat G$ since $A \hookrightarrow \hat G$.  Now consider the element $\hat \sigma^{p^n} \in \hat G$.  We know that it maps to the identity in $G$ and hence is an element of $A$.  Moreover it is fixed by the action of $\hat \sigma$.  Because $A^G = \oplus_{s=1}^t \langle \alpha_s^{(\sigma-1)^{\ell(s)-1}} \rangle$, there exists an element $\ve[c] \in \F_p^t$ so that $$\hat \sigma^{p^n} = \prod_{s=1}^t \alpha_s^{c_s (\sigma-1)^{\ell(s)-1}}.$$
%\end{proof}

It is worth noting that the relations on $\mathfrak{G}(M,\ve[0])$ are clearly the same as those for $M \rtimes G$, and hence these two groups (and group-theoretic embedding problems) are identical.  

As a first step towards determining when two embedding problems of this form are isomorphic, we have the following

\begin{lemma}\label{le:replace.free.with.zero}
Suppose that $M = \oplus_{i=1}^{\rk(M)} \langle \alpha_i \rangle$, and let $\ve[c] \in \F_p^{\rk(M)}$ be given.  If $\ve[d]$ is the vector such that 
$$d_i = \left\{\begin{array}{ll}0,&\mbox{ if }\ell(\alpha_i) = p^n\\c_i,&\mbox{ otherwise},\end{array}\right.$$ then $\mathfrak{G}(M,\ve[c])$ and $\mathfrak{G}(M,\ve[d])$ are isomorphic as group-theoretic embedding problems.
\end{lemma}

\begin{proof}
Consider the lift of $\sigma$ given by $\tilde \sigma = \left(\sum_{\ell(\alpha_i) = p^n} -c_i \cdot \alpha_i\right) \hat \sigma$. Applying relations (1)-(3) inductively, one can show that 
\begin{equation*}
\begin{split}
\tilde \sigma^{p^n}  = \left(\left(\sum_{\ell(\alpha_i) = p^n} -c_i\cdot \alpha_i\right) \hat \sigma\right)^{p^n} %&= \prod_{\ell(i) = p^n} \alpha_i^{-c_i}\hat \sigma \prod_{\ell(i) = p^n} \alpha_i^{-c_i}\hat \sigma \left(\prod_{\ell(i) = p^n} \alpha_i^{-c_i}\hat \sigma\right)^{p^n-2} \\
%&= \prod_{\ell(i) = p^n} \alpha_i^{-c_i(1+\sigma)} \hat \sigma^2 \left(\prod_{\ell(i) = p^n} \alpha_i^{-c_i}\hat \sigma\right)^{p^n-2} \\
%&= \prod_{\ell(i) = p^n} \alpha_i^{-c_i(1+\sigma+\sigma^2)} \hat \sigma^3 \left(\prod_{\ell(i) = p^n} \alpha_i^{-c_i}\hat \sigma\right)^{p^n-3} \\
% &\hspace{1.5in}\vdots \\
&=\left( \sum_{\ell(\alpha_i) = p^n} -c_i(1+\sigma+\cdots+\sigma^{p^n-1})\cdot \alpha_i\right) \hat \sigma^{p^n}
\end{split}
\end{equation*}
Since $\sum_{i=0}^{p^n-1} \sigma^i = \Psi^{p^n-1} \mod{p}$, we therefore have
\begin{equation*}
\begin{split}
\tilde \sigma^{p^n} &= \sum_{\ell(\alpha_i) = p^n} -c_i \Psi^{p^n-1}\cdot \alpha_i + \sum_{i=1}^{\rk(A)} c_i \Psi^{\ell(\alpha_i)-1}\cdot \alpha_i
\\&= \sum_{\ell(\alpha_i) \neq p^n} c_i \Psi^{\ell(\alpha_i)-1} \cdot \alpha_i.
\end{split}\end{equation*}
Because $\mathfrak{G}(M,\ve[c])$ can be generated by $\tilde \sigma, \alpha_1, \cdots, \alpha_{\rk(M)}$ so that the relations for $\mathfrak{G}(M,\ve[d])$ are satisfied, it must be that these two groups are isomorphic.  Since the projections of these groups onto $G$ are compatible, these two group-theoretic embedding problems are isomorphic.
\end{proof}

\begin{lemma}
Suppose that $M = \oplus_{i=1}^{\rk(M)} \langle \alpha_i \rangle$, and let $\ve[c] \in \F_p^{\rk(M)}$ be given.  Suppose that one can choose $1 \leq l \leq \rk(M)$ so that $\ell(\alpha_l)$ is minimal subject to the condition that $c_l \neq 0$ and $\ell(\alpha_l)<p^n$. Let $\ve[e_l]$ be the $l$th standard basis vector.  Then $\mathfrak{G}(M,\ve[c])\simeq \mathfrak{G}(M,\ve[e_l])$ as group-theoretic embedding problems over $G$.  On the other hand, if no such $l$ exists, then $\mathfrak{G}(M,\ve[c]) \simeq \mathfrak{G}(M,\ve[0])$ as group-theoretic embedding problems over $G$.
\end{lemma}

\begin{proof}
Lemma \ref{le:replace.free.with.zero} allows us to assume $\ve[c]$ to have $0$ coordinate in those positions $j$ corresponding to $\ell(\alpha_j) = p^n$.  If no such $l$ exists as in the statement of the theorem, then we appeal to the previous lemma to conclude that $\mathfrak{G}(M,\ve[c]) \simeq \mathfrak{G}(M,\ve[0])$ as group-theoretic embedding problems.

Now suppose that $l$ is chosen as in the statement of the theorem, and that $c_j = 0$ when $\ell(\alpha_j) = p^n$.  Define $\beta_j = \alpha_j$ for every $j \neq l$, and let $$\beta_l = c_l \cdot \alpha_l + \sum_{c_j \neq 0} c_j \Psi^{\ell(\alpha_j)-\ell(\alpha_i)}\cdot \alpha_j.$$  It is obvious that $\beta_l \in \oplus_{i=1}^{\rk(M)} \langle \alpha_i \rangle$ and $\alpha_l \in \oplus_{i=1}^{\rk(M)}\langle \beta_i\rangle$, so that $\{\beta_i\}_{i=1}^{\rk(M)}$ generates the same $\F_p[G]$-module as $\{\alpha_i\}_{i=1}^{\rk(M)}$.   It is equally clear that $\ell(\beta_i) = \ell(\alpha_i)$, and that $${\Psi^{\ell(\beta_l)-1}}\cdot\beta_l = {\Psi^{\ell(\alpha_l)-1}}\cdot \left({c_l}\cdot \alpha_l + \sum_{c_j \neq 0} {c_j\Psi^{\ell(\alpha_j)-\ell(\alpha_l)}}\cdot\alpha_j\right) = \hat \sigma^{p^n}.$$  
Hence $\mathfrak{G}(M,\ve[c])$ satisfies the relations defining $\mathfrak{G}(M,\ve[e_l])$.  Since the projections of these two groups onto $G$ are compatible, they are isomorphic as group-theoretic embedding problems.
\end{proof}

In light of the previous theorem, we see that the defining characteristics for an extension of $G$ by a finite $\F_p[G]$-module $M$ are the isomorphism type of $M$, together with the smallest length for a generator of $M$ which appears in relation (4) from Definition \ref{def:initial.group.extension}. Hence we introduce a new (and simpler) notation to keep track of the extensions of $G$ by $M$.

\begin{definition}
Suppose that $M = \oplus_{i=1}^{\rk(M)} \langle \alpha_i \rangle$, and let $1 \leq \mu < p^n$ be given so that there exists some $1 \leq l \leq \rk(M)$ with $\ell(\alpha_l) = \mu$.  Then we define $M \bullet_\mu G$ to be the group $\mathfrak{G}(M,\ve[e_l])$.  We define $M \bullet_{p^n} G$ to be $\mathfrak{G}(M,\ve[0]) \simeq M \rtimes G$, regardless of whether $M$ contains a summand of dimension $p^n$.
\end{definition}

\begin{remark*}
Of course our definition for $M \bullet_\mu G$ hasn't specified which of the various $l$ satisfying $\ell(\alpha_l) = \mu$ should be chosen in the definition.  This small ambiguity is trivial to resolve and doesn't disturb the isomorphism class of the group-theoretic problem in question, so we won't address it.  One might also point out that the definition of $\mathfrak{G}(M,\ve[e_l])$ requires us to name generators for $M$, though we won't always be so careful to do so (nor will we worry about ambiguities that arise when we don't specifically list generators).  In those times where this specificity is important, we will name generators for $M$ explicitly.
\end{remark*}

\begin{theorem}\label{th:counting.isom.types.of.group.embedding.problems}
For a finite $\F_p[G]$-module $M$, there are $\rknf(M)+1$ many isomorphism types for embedding problems over $G$ with kernel $M$: one is $M \rtimes G$, and the others correspond to $M \bullet_\mu G$ where $1\leq \mu < p^n$ ranges over the dimensions of non-free summands in an $\F_p[G]$-decomposition of $M$.
\end{theorem}

\begin{remark*}
This theorem was shown in the case that $M = \langle \alpha \rangle$ is a cyclic submodule by Waterhouse in \cite{Wa}.  In the case where the cyclic submodule isn't isomorphic to $\F_p[G]$, the two possibilities are given by the semi-direct product and another group which --- in our notation --- is $\langle \alpha \rangle \bullet_{\ell(\alpha)} G$.
\end{remark*}

\begin{proof}
As usual, write $M = \oplus_{i=1}^{\rk(M)}\langle \alpha_i \rangle$.  We have already shown that any such group-theoretic embedding problem is isomorphic to $\mathfrak{G}(M,\ve[c])$ for some $\ve[c] \in \F_p^{\rk(M)}$, and that this group is isomorphic to either $M \rtimes G = M \bullet_{p^n} G$ or $M \bullet_\mu G$ for some $1 \leq \mu <p^n$ such that there exists $1 \leq i \leq \rk(M)$ with $\mu = \ell(\alpha_i)$. Now we must show that if $\mu_1 \neq \mu_2$, then $M \bullet_{\mu_1} G$ and $M \bullet_{\mu_2} G$ are not isomorphic as group-theoretic embedding problems.

First, consider the case $\mu_1 = p^n$ and $\mu_2 < p^n$; our strategy will be to count the number of elements of order greater than $p^n$ in both groups.  If we take an arbitrary element's $p^n$th power in either group, we find
\begin{equation}\label{eq:computing.elements.of.order.pn}
\begin{split}
\left(\left(\sum_{i} {f_i(\sigma)}\cdot \alpha_i\right) \hat \sigma^j\right)^{p^n} &= \left(\sum_{i} {(1+\sigma^j + \cdots + (\sigma^j)^{p^n-1})f_i(\sigma)}\cdot \alpha_i\right) (\hat \sigma^j)^{p^n} \\
&= \left(\sum_i {(\sigma^{j}-1)^{p^n-1}f_i(\sigma)}\cdot \alpha_i\right)(\hat \sigma^{p^n})^{j}.
\end{split}
\end{equation}
In either group, if $j=p^k h$ for some $k>0$, then this element is trivial: certainly $(\hat \sigma^{p^n})^{j}$ is trivial since $\hat \sigma$ has order at most $p^{n+1}$, and we also have
$$\sum_{s=1}^{p^n-1} (\sigma^j)^s \equiv (\sigma^j-1)^{p^n-1} \equiv (\sigma^h-1)^{p^{n+k}-p^k} \in \langle \Psi^{p^n} \rangle \subseteq \F_p[G].$$
Now when $(j,p)=1$, equation (\ref{eq:computing.elements.of.order.pn}) becomes 
\begin{equation*}\begin{split}
\left(\left(\sum_{i} {f_i(\sigma)}\cdot \alpha_i\right) \hat \sigma^j\right)^{p^n} %&= c \cdot \left(\sum_i {(\sigma-1)^{p^n-1}f_i(\sigma)}\cdot \alpha_i\right)(\hat \sigma^{p^n})^{j} \\
&= c\cdot \left(\sum_{\ell(\alpha_i)=p^n} {\Psi^{p^n-1}f_i(\sigma)}\cdot \alpha_i\right)(\hat \sigma^{p^n})^j
\end{split}\end{equation*} where $c$ is the multiplicative inverse of $j$ in $\F_p^\times$.  In the group $M \bullet_{p^n} G$ the term $(\hat \sigma^{p^n})^j$ vanishes, and so the term is nonzero only when at least one $f_i(\sigma) \not\in \langle \Psi\rangle \subseteq \F_p[G]$.  In $M\bullet_{\mu_2} G$, however, the term $(\hat \sigma^{p^n})^j$ is nonzero and independent from the terms in the sum, and hence this element is nonzero for all $f_i(\sigma)\in  \F_p[G]$.  Hence $M \bullet_{\mu_2} G$ has more elements of order $p^{n+1}$ than $M \bullet_{p^n} G$, so these two groups are not isomorphic.

With this case resolved, suppose without loss that $\mu_1 < \mu_2<p^n$. The defining relation for $M \bullet_{\mu_1} G$ is that we can find a lift $\hat \sigma$ and a generator $\alpha_i \in M$ with $\ell(\alpha_i) = \mu_1$ so that $$\hat \sigma^{p^n} = {\Psi^{\mu_1-1}}\alpha_i.$$  Note that this same equation holds true if we quotient by the normal subgroup $\langle \alpha_j \rangle_{j \neq i}$, and hence we have a surjection of group-theoretic embedding problems: 
$$\xymatrix{
M \bullet_{\mu_1} G \ar@{->>}[r] \ar@{->>}[d]& G \ar@{=}[d]\\
\F_p[G]/\langle \Psi^{\mu_1}\rangle \bullet_{\mu_1} G \ar@{->>}[r] & G
}.$$ 

On the other hand, consider the group $M \bullet_{\mu_2} G$; the defining relation for this group tells us we can choose a lift $\tilde \sigma$ for $\sigma$ and a generator $\beta_j$ of $M$ with $\ell(\beta_j) = \mu_2$ and so that $$\tilde \sigma^{p^n} = {\Psi^{\mu_2-1}}\cdot \beta_j.$$  Now suppose we have a surjection of group-theoretic embedding problems of the form
$$\xymatrix{
M \bullet_{\mu_2} G \ar@{->>}[r] \ar@{->>}[d]_{\psi}& G \ar@{=}[d]\\
\F_p[G]/\langle \Psi^{\mu_1}\rangle \bullet_{\lambda} G \ar@{->>}[r] & G,
}$$ where $\lambda\in \{\mu_1,p^n\}$.  (I.e., any group-theoretic embedding problem over $G$ whose kernel is the module $\F_p[G]/\langle \Psi^{\mu_1}\rangle$.) The five lemma tells us that these arise from quotients of $M$ isomorphic to $\F_p[G]/\langle \Psi^{\mu_1}\rangle$ --- i.e., from submodules $S$ of $M$ so that $M/S \simeq \F_p[G]/\langle \Psi^{\mu_1}\rangle$.  Notice that for any such $S$ we have ${\Psi^{\mu_2-1}} \cdot \beta_j\in S$, since otherwise $M/S$ would contain a cyclic submodule generated by $\beta_j$ that has length at least $\mu_2>\mu_1$.  

Now notice that $\psi(\tilde \sigma)$ is a lift of $\sigma$ in $\F_p[G]/\langle \Psi^{\mu_1}\rangle \bullet_\lambda G$.  Because $$\tilde \sigma^{p^n} = {\Psi^{\mu_2-1}}\cdot \beta_j$$ within $M \bullet_{\mu_2} G$, we therefore have $\psi(\tilde\sigma)^{p^n} = 0 \in M/S$.  Hence it follows that $\lambda=p^n$.  Since $\F_p[G]/\langle \Psi^{\mu_1}\rangle \bullet_{\mu_1} G \not\simeq \F_p[G]/\langle \Psi^{\mu_1}\rangle \bullet_{p^n} G$, the result follows.
\end{proof}

\begin{proposition}\label{prop:embedding.problems.with.different.kernels.are.different}
If $A$ and $M$ are non-isomorphic finite $\F_p[G]$-modules, then any group-theoretic embedding problem $\xymatrix{\hat G_M \ar@{->>}[r] & G}$ with kernel $M$ is not isomorphic to any group-theoretic embedding problem $\xymatrix{\hat G_A \ar@{->>}[r] & G}$ with kernel $A$.
\end{proposition}

\begin{proof}
Let $j$ be the largest number such that $|M_{\{j\}}| \neq |A_{\{j\}}|$, and assume that $|A_{\{j\}}|>|M_{\{j\}}|$. (Here we use the notation from Equation (\ref{eq:subbrace.notation}).)  Then there exists a surjection$\xymatrix{A \ar@{->>}[r] & \left(\F_p[G]/\langle \Psi^j\rangle\right)^{D_j}}$ where $D_j = \dim_{\F_p}(A_{\{j\}})$, whereas $M$ has no such surjection.  Hence we have a surjection of embedding problems  
$$\xymatrix{
\hat G_A \ar@{->>}[r] \ar@{->>}[d] & G \ar@{=}[d]\\
\left(\frac{\F_p[G]}{\langle \Psi^j\rangle}\right)^{D_j} \bullet_{\mu} G \ar@{->>}[r] & G}$$ but no such surjection for $\hat G_M$.
\end{proof}

\begin{example}
For the sake of concreteness, we finish this section by listing all the groups of order $p^4$ which can be realized as an extension of a cyclic $p$-group $G$ by an $\F_p[G]$-module $M$.  Obviously there is only one such group with $G \simeq \Z/p^4\Z$, and no such groups when $G \simeq \Z/p^k\Z$ for $k>4$.

When $G \simeq \Z/p^3\Z$ it must be the case that $M$ is $1$-dimensional.  Hence $M \simeq \F_p$, and there are two possible extensions: $\F_p \bullet_{1} \Z/p^3\Z$ and $\F_p \bullet_{p^3} \Z/p^3\Z$.  The former is a nonsplit extension of $\Z/p^3\Z$ by $\Z/p\Z$ which must be abelian (since the kernel has trivial $G$-action), and so must be $\Z/p^4\Z$; the latter is a split extension of $\Z/p^3\Z$ by $\Z/p\Z$ which is abelian, and so must be $\Z/p^3\Z \times \Z/p\Z$.

When $G \simeq \Z/p^2\Z$ we have two possible structures for $M$: either $M \simeq \F_p \oplus \F_p$ or $M \simeq \F_p[G]/\langle \Psi^2\rangle$.  Since in each case there is only one non-free isomorphism class of indecomposable in the structure of $M$, we have four possible groups: 
\begin{equation*}\begin{split}
\left(\F_p \oplus \F_p\right) \bullet_1  \Z/p^2\Z , \quad%&\simeq \Z/p^4\Z \times \Z/p\Z\\
\left(\F_p \oplus \F_p \right) \bullet_{p^2} \Z/p^2\Z,\quad %&\simeq \Z/p\Z \times \Z/p\Z \times \Z/p^3\Z\\
\left(\F_p[G]/\langle \Psi^2\rangle\right) \bullet_{2}\Z/p^2\Z,\quad  %& \\
\left(\F_p[G]/\langle \Psi^2\rangle\right) \bullet_{p^2}\Z/p^2\Z.% &\simeq \left(\F_p[G]/\langle \Psi^2\rangle\right) \rtimes \Z/p^3\Z.
\end{split}
\end{equation*}
One can show that the first group in this list is isomorphic to $\Z/p^3\Z \times \Z/p\Z$, and hence to $\F_p \bullet_{p^3} \Z/p^3\Z$ from the previous case.  Likewise the second group is isomorphic to $\Z/p^2\Z \times \Z/p\Z \times \Z/pZ$, and of course $(\F_p[G]/\langle \Psi^2\rangle)\bullet_{p^2}\Z/p^2\Z \simeq \F_p[G]/\langle \Psi^2 \rangle \rtimes \Z/p^2\Z$. This last group is also studied in \cite{M1} under the name $G_3$: 
$$G_3 = \left\langle g_1,g_2,g_3,g_4: g_1^p=g_4; g_2^p=g_3^p=g_4^p=1; g_2g_1=g_1g_2g_3; [g_3,g_i]=[g_4,g_i]=1 \mbox{ for all }i\right\rangle.$$  We can identify $(\F_p[G]/\langle \Psi^2\rangle)\bullet_{p^2}\Z/p^2\Z$ with $G_3$ by setting $g_1 = \hat \sigma^{-1}$, $g_2 = 1 \in \F_p[G]/\langle \Psi^2\rangle$, $g_3 = \Psi \in \F_p[G]/\langle \Psi^2 \rangle$ and $g_4=\hat \sigma^{-p}$.  Finally, $\left(\F_p[G]/\langle \Psi^2\rangle\right) \bullet_{2}\Z/p^2\Z$ is also studied in \cite{M1} under the name $G_5$:
$$G_5=\langle g_1,g_2,g_3,g_4: g_1^p=g_3;g_3^p=g_4;g_2^p=g_4^p=1;g_2g_2=g_1g_2g_4;[g_3,g_i]=[g_4,g_i]=1 \mbox{ for all }i\rangle.$$
We can identity $G_5$ with $\mathfrak{G}(\F_p[G]/\langle \Psi^2\rangle,-\ve[e_1])$ by setting $g_1=\hat \sigma^{-1}, g_2=1 \in \F_p[G]/\langle \Psi^2 \rangle, g_3=\hat \sigma^{-p},g_4=\Psi \in \F_p[g]/\langle \Psi^2 \rangle$.  This is sufficient since we have already shown $\F_p[G]/\langle \Psi^2 \rangle \bullet_2 \Z/p^2\Z\simeq \mathfrak{G}(\F_p[G]/\langle \Psi^2\rangle,-\ve[e_1])$.

%In particular, some groups can be realized as an extension of a cyclic $p$-group $G$ by an $\F_p[G]$-module in more than one way, provided that the group $G$ is allowed to vary.  This sits in contrast to Theorem \ref{th:counting.isom.types.of.group.embedding.problems} and Proposition \ref{prop:embedding.problems.with.different.kernels.are.different} which say that such extensions must be distinct when the group $G$ is fixed.

When $G \simeq \Z/p\Z$ there are more possible structures for $M$: $\oplus_3 \F_p$, $\F_p \oplus \F_p[G]/\langle \Psi^2\rangle$, and $\F_p[G]/\langle \Psi^3\rangle$.  The corresponding groups are then
\begin{gather*}
\left(\oplus_3 \F_p\right) \bullet_1 \Z/p\Z ,\quad %&\simeq \Z/p^3\Z \times \Z/p\Z \times \Z/p\Z\\
\left(\oplus_3 \F_p\right) \bullet_{p} \Z/p\Z ,\quad %&\simeq \Z/p^2\Z \times \Z/p\Z \times \Z/p\Z \times \Z/p\Z\\
\left(\F_p \oplus \F_p[G]/\langle \Psi^2 \rangle\right) \bullet_{1} \Z/p\Z ,\quad %& \\
\left(\F_p \oplus \F_p[G]/\langle \Psi^2 \rangle\right) \bullet_{2} \Z/p\Z ,\quad \\%& \\
\left(\F_p \oplus \F_p[G]/\langle \Psi^2 \rangle\right) \bullet_{p^2} \Z/p\Z ,\quad %&\\
\left(\F_p[G]/\langle \Psi^3 \rangle\right) \bullet_{3} \Z/p\Z,\quad %&\\
\left(\F_p[G]/\langle \Psi^3 \rangle\right) \bullet_{p} \Z/p\Z.%&
\end{gather*}

\end{example}

\section{Elementary $p$-abelian extensions of fields}\label{sec:embedding.problems.in.terms.of.classical.parameterizing.spaces}

Recall our standing notations: $G=\Gal(K/F)  = \langle \sigma \rangle \simeq \Z/p^n\Z$ and $\Psi:=\sigma-1$.  If we are interested in computing $\Gal(L/F)$ when $L/K$ is an elementary $p$-abelian extension which is additionally Galois over $F$, then the preceding section tells us that we need to understand the module structure of $\Gal(L/K)$ together with a value for $\hat \sigma^{p^n}$, where $\hat \sigma \in \Gal(L/F)$ is a lift of $\sigma \in \Gal(K/F)$.  In this section we will consider how to determine these properties in terms of the classic parameterizing spaces for elementary $p$-abelian extensions.

\subsection{Classifying elementary $p$-abelian extensions}
Parametrizing spaces for elementary $p$-abelian extensions of fields have been known for quite some time.  When $\ch{K} \neq p$, let $\hat K = K(\xi_p)$ and $\hat F = F(\xi_p)$.  Note that $([\hat F: F],p) = 1$, and hence $\Gal(\hat K/\hat F)\simeq G$.  The Galois group of $\hat F/F$ is cyclic, and for a generator $\epsilon$ we write $\epsilon(\xi_p) = \xi_p^t$.  Relative Kummer theory tells us that the $\oplus^k \Z/p\Z$-extensions of $\hat K$ correspond to $k$-dimensional $\F_p$-subspaces of $\hat K^\times/\hat K^{\times p}$ which are in the $t$-eigenspace of $\epsilon$; we can then recover $(\oplus^k \Z/p\Z)$-extensions of $K$ via descent.  The correspondence between an $\F_p$-subspace $M$ within $\hat K^\times/\hat K^{\times p}$ and an extension $L/K$ is given explicitly by 
\begin{equation*}\begin{split}
M &\mapsto \mbox{ the maximal $p$-extension of $K$ in }\hat K\left(\root{p}\of{m}: m \in M\right)
\\ L &\mapsto \mbox{ the $t$-eigenspace of $\epsilon$ within }\frac{L(\xi_p)^{\times p}\cap \hat K^\times}{\hat K^{\times p}}.%\right|_{\epsilon=t}.
\end{split}\end{equation*}

In the case where $\ch{K}=p$, the parametrizing space is given to us by Artin-Schreier theory, which says that elementary $p$-abelian extensions of $K$ are given by $\F_p$-subspaces of $K/\wp(K)$, where $\wp(K) = \{k^p-k: k \in K\}.$  For $k \in K$ we write $\rho(k)$ to denote a root of the equation $x^p-x-k$.  Using this notation, the correspondence is given by 
\begin{equation*}\begin{split}M &\mapsto K\left(\rho(m): m \in M\right)\\L &\mapsto \frac{\wp(L) \cap K}{\wp(K)}.\end{split}\end{equation*}

%Perhaps slightly less well known is the remaining parametrizing space which describes elementary $p$-abelian extensions of $K$ when $\ch{K} \neq p$, yet $K$ doesn't contain a primitive $p$th root of unity.  In this case, let $\epsilon$ be a generator of $K(\xi_p)^/K$, and suppose that $\epsilon(\xi_p) = \xi_p^e$.  Then the appropriate parametrizing space is the subspace of $K(\xi_p)^\times/K(\xi_p)^{\times p}$ on which $\epsilon$ acts by exponentiation by $t$. (For more details, see \ref{}.) Our correspondence is therefore 
%Notice that one can phrase Kummer Theory in precisely this same language, since in this case we have $K(\xi_p) = K$.  We will use this convention for the rest of the paper.

Regardless of the field $K$ under consideration, we will write $J(K)$ for the $\F_p$-space whose subspaces parametrize elementary $p$-abelian extensions; that is to say, we write $J(K)$ for the $t$-eigenspace of $\epsilon$ within $\hat K^\times/\hat K^{\times p}$ when $\ch{K} \neq p$, and we write $J(K)$ for $K/\wp(K)$ when $\ch{K}=p$.  When we consider $J(K)$ as an $\F_p[G]$-module we will follow our standard notation and write the group operation on $J(K)$ additively and the $\F_p[G]$-action multiplicatively.  There are two natural maps connecting $J(K)$ and $J(F)$ that will be important to us: $$N:J(K) \to J(F) \quad \mbox{ and } \quad \iota:J(F) \to J(K).$$  The first is induced by either $N_{\hat K/\hat F}:\hat K \to \hat F$ (when $\ch{K} \neq p$) or $\text{Tr}_{K/F}:K \to F$ (when $\ch{K}=p$), and the second is induced by the usual inclusion of fields $\hat F \hookrightarrow \hat K$ when $\ch{K} \neq p$ and $F \hookrightarrow K$ when $\ch{K}=p$).

By putting additional structure on $J(K)$, one can make $J(K)$ a classifying space for a broader range of groups.  In particular, we will focus on the $\F_p[G]$-structure of $J(K)$ and ask what it tells us about the Galois-ness of the extension $L/F$.  The primordial result in this vein is that an elementary $p$-abelian extension $L/K$ is Galois over $F$ if and only if the corresponding $M \subseteq J(K)$ is an $\F_p[G]$-module; this was mentioned in \cite{Wa} when $\ch{K} \neq p$, and the proof is straightforward in the $\ch{K}=p$ case as well.  

%For an $\F_p[G]$-module $A = \oplus_{i=1}^{\rk(A)} \langle \alpha_i \rangle$ and a number $\lambda$ such that there exists $1 \leq i \leq \rk(A)$ with $\lambda = \ell(\alpha_i)$, o
%Our ultimate goal is to show that one can parameterize all solutions to any embedding problem of the form $\xymatrix{M \bullet_{\mu} G \ar@{->>}[r]& G}$ by a particular collection of submodules of $J(K)$. 

\subsection{Computations for cyclic modules}

When $L/K$ is an elementary $p$-abelian extension that corresponds to a cyclic submodule in $J(K)$, Waterhouse was able to compute the structure of $\Gal(L/F)$.  The key ingredient in his analysis is to note that if $L$ is a particular finite elementary $p$-abelian extension of $K$ and $M$ is the corresponding $\F_p$-space in $J(K)$, then there is a $G$-equivariant perfect pairing 
\begin{equation*}\begin{split}
\Gal(L/K) \times M &\to \F_p
\end{split}\end{equation*}
which exhibits a duality between these two groups.  Since we know from Galois theory that the Galois groups fit into a short exact sequence 
$$\xymatrix{1 \ar[r] & \Gal(L/K) \ar[r] & \Gal(L/F) \ar[r] & \Gal(K/F) \ar[r] & 1},$$
we have that $\Gal(L/F)$ is an extension of $\Gal(K/F)$ by $\Gal(L/K)$.  But since $\Gal(L/K) \simeq \check{M}$, and since all finitely-generated $\F_p[G]$-modules are self-dual (see \cite[Sec.~1]{MSSauto}), one can interpret $\Gal(L/F)$ as an extension of $\Gal(K/F)$ by $M$.  All one needs to determine then is the value of $\hat \sigma^{p^n}$, which Waterhouse accomplishes using a particular field-theoretic computation on a generator for $M$.

Fortunately, the proofs carry over into the characteristic $p$ setting almost entirely unchanged, since they depend only on having a parametrizing space for elementary $p$-abelian extensions whose $\F_p[G]$-module theory encodes the property of being Galois over $F$, together with the $G$-equivariant Kummer pairing.  In the characteristic $p$ setting, the Kummer pairing is replaced with the analogous Artin-Schreier pairing, defined as follows. Note that any two roots of $x^p-x-k$ differ by an element of $\F_p$.  Furthermore, an element $\tau \in \Gal(L/K)$ acts by permuting roots of $x^p-x-k$, and hence we can define $$\langle \tau,m \rangle:=\tau\left(\rho(m)\right)-\rho(m).$$ %(Note that we have written the $\F_p[G]$-action additively since we are in characteristic $p$ and the underlying $\F_p$-structure for Artin-Schreier theory is on the additive group $K$.) 

\begin{proposition}\cite[Lemma 3.1]{BS}
Suppose that $\ch{K} = p$.  If $L/K$ is an elementary $p$-abelian extension and $M\subseteq J(K)$ is the corresponding submodule, then the Artin-Schreier pairing is $G$-equivariant and perfect.
\end{proposition}

Since the module structure of $\Gal(L/K)$ determined by the module structure of the corresponding submodule $M \subseteq J(K)$, we only need to determine how to calculate the value of $\hat \sigma^{p^n}$, where $\hat \sigma \in \Gal(L/F)$ is a lift of $\sigma \in \Gal(K/F)$.  

%Towards computing $\hat \sigma^{p^n}$, we give the following definition.  Recall that for an element $w$ of an $\F_p[G]$-module $W$, the quantity $\ell(w)$ measures $\dim_{\F_p}\langle w \rangle$.  This quantity determines the isomorphism type of the indecomposable submodule $\langle w \rangle$.  

\begin{definition*}
Let $\alpha \in J(K)$ be given so that $\ell(\alpha)<p^n$.  If $\ch{K} \neq p$ then the index of $\alpha$, written $e(\alpha)$ is defined by 
$$\xi_p^{e(\alpha)} = \root{p}\of{N_{\hat K/\hat F}(\alpha)}^{\Psi}.$$  If instead $\ch{K} = p$, then the index of $\alpha$ is defined to be $$e(\alpha) = \Psi \left(\rho\left(Tr_{K/F}(\alpha)\right)\right).$$  An element is said to have trivial index if its index is $0$.
\end{definition*}

\begin{remark*}The index can be reexpressed using the maps $N$ and $\iota$ we defined earlier.  Note first that the action of $\iota \circ N$ is equivalent to the $\F_p[G]$-action of $$1+\sigma+\cdots + \sigma^{p^n-1} \equiv \Psi^{p^n-1},$$ and hence the domain of the index function can be reexpressed: $$\{\alpha \in J(K): \ell(\alpha)<p^n\} = \ker(\iota\circ N).$$ Moreover, the index function is trivial precisely on the subset $\ker(N) \subseteq \ker(\iota \circ N)$; hence the index function is a way for detecting whether or not an element from $\ker(\iota \circ N)$ is in $\ker(N)$.
\end{remark*}

\begin{proposition}\label{prop:cyclic.group.computation}
Suppose that $\alpha \in J(K)$, and let $L$ be the extension corresponding to $\langle \alpha \rangle$.  Then 
\begin{itemize}
\item if $\ell(\alpha)=p^n$ or $e(\alpha) = 0$, then $\Gal(L/K) \simeq \F_p[G]/\langle \Psi^{\ell(\alpha)}\rangle \rtimes G$; and
\item if $\ell(\alpha)<p^n$ and $e(\alpha) \neq 0$, then $\Gal(L/K) \simeq  \F_p[G]/\langle \Psi^{\ell(\alpha)}\rangle \bullet_{\ell(\alpha)} G$. 
\end{itemize}
\end{proposition}

\begin{proof}
The result is precisely \cite[Prop.~2]{MSS1} when $\ch{K} \neq p$.  The proof in this case only relies on a $G$-equivariant, perfect pairing between $\Gal(L/K)$ and $\langle \alpha \rangle$, and since such a pairing is provided when $\ch{K} = p$ above, the result also follows.  For the sake of concreteness, though, we show the reader how one goes about verifying this identity more directly in the case $\ch{K}=p$; of course, this same idea also applies when $\ch{K}\neq p$ after minor notational changes.  

Suppose that $\ch{K}=p$ and consider $\langle \alpha \rangle \subseteq J(K)$; let $L/F$ be the corresponding extension of fields.  Recall that there is only one extension of $G$ by $\F_p[G]$, and hence if $\ell(\alpha) = p^n$ then we have $\Gal(L/K) \simeq \langle \alpha \rangle \simeq \F_p[G]$ and $\Gal(L/F) \simeq \F_p[G] \rtimes G$.  

Suppose, then, that $\ell(\alpha)<p^n$.  If $\hat \sigma \in \Gal(L/F)$ is a lift of $\sigma \in \Gal(K/F)$, then we need to determine whether or not $\hat \sigma^{p^n}$ is trivial.  Recall that $\hat \sigma^{p^n} \in \Gal(L/K)$ is in the submodule of elements fixed by the action of $\sigma$. The generator of the fixed module in $\Gal(L/K)$ is dual to the element $\alpha$, and so we simply need to know whether $\hat \sigma^{p^n}$ acts trivially on $\rho(\alpha)$ or not; i.e., we need to compute $\langle \hat \sigma^{p^n},\alpha\rangle$.  Since $(\sum_{i=1}^{p^n} \hat \sigma^i)\rho(\alpha)$ and $\rho(Tr_{K/F}(\alpha))$ are both roots for the same polynomial, they have the same image under $\hat \sigma-1$; hence we have
\begin{equation*}\begin{split}
\langle \hat \sigma^{p^n},\alpha\rangle &= \hat \sigma^{p^n}\left(\rho(\alpha)\right) - \rho(\alpha)\\
&=\left(\hat\sigma^{p^n}-1\right)\rho(\alpha)\\
&=(\hat \sigma-1) \left(\sum_{i=0}^{p^n-1} \hat \sigma^i\right) \rho(\alpha)\\
&= (\hat \sigma-1) \left(\rho(Tr_{K/F}(\alpha)\right).
\end{split}\end{equation*}  Since $\ell(\alpha)<p^n$ it follows that $Tr_{K/F}(\alpha) = \Psi^{p^n-1}\alpha \in \wp(K)$, and so $\rho(Tr_{K/F}(\alpha)) \in K$.  Hence the action of $\hat \sigma$ on this element is identical to the action of $\sigma$, and so $$\langle \hat \sigma^{p^n},\alpha \rangle = \Psi \left(\rho(Tr_{K/F}(\alpha))\right) = e(\alpha).$$ Hence $\hat \sigma^{p^n}$ is trivial if and only if $e(\alpha) = 0$.
\end{proof}

%We have already seen from Lemma \ref{} that if $\ell<p^n$ and $i,j \in \F_p^\times$,  then groups $\mathfrak{G}(A_\ell,i)$ and $\mathfrak{G}(A_\ell,j)$ are isomorphic. One can therefore say that the distinguishing characteristic for determining the group structure of $\Gal(L/F)$ is whether $M$ is generated by an element with non-trivial index.

%An important detail to note from \cite[Prop.~]{} is that $\mathfrak{G}(A_{p^n},i) \simeq \mathfrak{G}(A_{p^n},j)$ for every $i,j \in \F_p$.  This same result is not true for $A_\ell$ when $\ell<p^n$, since $\mathfrak{G}(A_\ell,0) \not\simeq \mathfrak{G}(A_\ell,c)$ for $c \in \F_p^\times$.  However, one does have the following

%\begin{corollary}
%If $c_1,c_2 \in \F_p^\times$ and $\ell<p^n$, then $\mathfrak{G}(A_\ell,c_1) \simeq \mathfrak{G}(A_\ell,c_2)$.
%\end{corollary}
%\begin{proof}
%Suppose that $e(m) = c_1$, where $\langle m \rangle \simeq A_\ell$.  Then $\Gal(L/K) \simeq \mathfrak{G}(A_\ell,c_1)$ by Proposition \ref{}.  Now let $u \in \F_p^\times$ satisfy $uc_1 = c_2$.  Then we have $\langle m^u \rangle = \langle m \rangle$; since $e(m^u) = ue(m) = c_2$, this implies $\Gal(L/K) \simeq \mathfrak{G}(A_\ell,c_2)$. 
%\end{proof}

\subsection{Moving beyond cyclic modules}
Now that we have a description of Galois groups that arise from cyclic submodules of $J(K)$, we can determine the Galois structure of a generic extension of $G$ by a finite $\F_p[G]$-module in terms of its module structure and the index.  

\begin{definition}
Suppose that $M \subseteq J(K)$.  If there exists an element $m \in M$ with $\ell(m)<p^n$ such that $e(m) \neq 0$, then we define $$\lambda(M)=\min_{m \in M} \{\ell(m):  \ell(m)<p^n \mbox{ and } e(m) \neq 0\}.$$ Otherwise, we define $\lambda(M) = p^n$.  
\end{definition}

\begin{remark*}
Our definition of $\lambda(M)$ is written piecewise, but it is possible to combine these two cases into the single expression 
$$\lambda(M) = \inf_{m \in M} \{\ell(m): m \in \ker(\iota \circ N)\setminus \ker(N)\}.$$
This follows because we can replace the length and index assumptions with appropriate statements concerning $\ker(\iota\circ N)$ and $\ker(N)$, and because the natural interpretation for $\inf(\emptyset)$ in this context is $\inf(\emptyset) = p^n$ (since the quantities of interest are dimensions of cyclic submodules, and hence values between $0$ and $p^n$).
\end{remark*}

We are now prepared to prove the main result of this paper.

%\begin{theorem}\label{th:parameterization}
%Suppose that $A \subseteq J(K)$ is a finitely generated $\F_p[G]$-module and $L/F$ is the corresponding extension.  Let $M$ be a finitely generated $\F_p[G]$-module, and $\mu$ either $0$ or the $\F_p$-dimension of some summand of $M$. Then $L$ solves the embedding problem $\xymatrix{M \bullet_\mu G \ar@{->>}[r]& G}$ over $K/F$ if and only if $A \simeq M$ and $\mu = \lambda(A)$.
%\end{theorem}

\begin{proof}[Proof of Theorem \ref{th:main.theorem}]
Let $M \subseteq J(K)$ be a given $\F_p[G]$-module, and $L/F$ its corresponding extension.  We have already seen that $\Gal(L/F)$ can be viewed as an extension of $\Gal(K/F)$ by $M$, and so Theorem \ref{th:counting.isom.types.of.group.embedding.problems} tells us that  $\Gal(L/F)$ is isomorphic to $\xymatrix{M \bullet_\mu G \ar@{->>}[r]& G}$ for some $1 \leq \mu \leq p^n$.  Our goal is to show that $\mu = \lambda(M)$.  Once we have done this, Proposition \ref{prop:embedding.problems.with.different.kernels.are.different} tells us that this is the only embedding problem that $L$ solves over $K/F$.

First, if $\lambda(M) = p^n$, then there are no elements of length less than $p^n$ with nontrivial index in $M$.  We claim then that $\Gal(L/F) \simeq M \rtimes G$.  To see this is true, let $M = \oplus \langle \alpha_i \rangle$, and define $L_i$ to be the field corresponding to $\langle \alpha_i \rangle$; let $\tau_i \in \Gal(L/F)$ be chosen so that $\tau_i$ restricts to the trivial automorphism on all extensions $L_j/K$ for $j \neq i$, and which restricts to an $\F_p[G]$-generator of $\Gal(L_i/K)$. By Proposition \ref{prop:cyclic.group.computation} we know that $\Gal(L_i/F) \simeq \langle \overline{\tau_i} \rangle \rtimes G$ since either $e(\alpha_i) = 0$ or $\ell(\alpha_i) = p^n$. We also have a surjection of embedding problems that comes from Galois theory, and which corresponds to quotienting by the subgroup $\langle \{\tau_j\}_{j \neq i} \rangle$:
$$\xymatrix{
M \bullet_\mu G \ar@{->>}[d]_\psi \ar@{->>}[r] & G \ar@{=}[d]\\
\langle \overline{\tau_i} \rangle \rtimes G \ar@{->>}[r] & G.
}$$
Now if $\hat \sigma \in M \bullet_\mu G$ is a lift of $\sigma$, then $\psi(\hat \sigma) \in \langle \overline{\tau_i} \rangle  \rtimes G$ is a lift of $\sigma$, and we know that $\psi(\hat \sigma)^{p^n} = 0$ if $\ell(\alpha_i) < p^n$.  Hence we have
$$\hat \sigma^{p^n} = \sum_{\ell(\alpha_i) = p^n} {c_i \Psi^{p^n-1}}  \tau_i,$$ and it follows that $\Gal(L/F) \simeq M \rtimes G$.

Suppose, then, that $\lambda(M)<p^n$.  We begin by choosing a decomposition $M = \oplus_{i=1}^{\rk(M)} \langle \alpha_i \rangle$ satisfying
$$\left\{\begin{array}{l}e(\alpha_1) \neq 0 \mbox{ and }\ell(\alpha_1) = \lambda(M)\\e(\alpha_i)=0 \mbox{ for all }i>1 \mbox{ with }\ell(\alpha_i)<p^n.\end{array}\right.$$  To see this is possible, note that if we have any decomposition $M = \oplus_{i=1}^{\rk(M)} \langle \beta_i \rangle$, then we choose $j$ such that $\ell(\beta_j)$ is minimal amongst all elements with $e(\beta_j) \neq 0$.   For convenience we can assume that $j=1$ and also that $e(\beta_1) = 1$.  We then define $\alpha_i = \beta_i$ if either $i=1$ or $\ell(\beta_i) = p^n$, and for $i>1$ with $\ell(\beta_i)<p^n$ we set $\alpha_i = {-e(\beta_i)} \cdot \beta_1 + \beta_i$.  It is easy to check that $\oplus \langle \alpha_i \rangle = \oplus \langle \beta_i \rangle$, and that $\lambda(M) = \ell(\alpha_1)$.  

Let $L_i$ be the extension corresponding to $\langle \alpha_i \rangle$; as before, let $\tau_i \in \Gal(L/F)$ be chosen so that $\tau_i$ restricts to the trivial automorphism on all extensions $L_j/K$ for $j \neq i$, and which restricts to an $\F_p[G]$-generator of $\Gal(L_i/K)$.  We know that $\Gal(L_i/K)$ solves the embedding problem $\langle \overline{\tau_i} \rangle \bullet_{\mu_i} G$ over $K/F$, with $\mu_1 = \lambda(M)$ and $\mu_i = p^n$ for $i>1$.  For each $1 \leq i \leq \rk(M)$ we have a surjection of embedding problems that arises by quotienting by the subgroup generated by $\langle\{\tau_j\}_{j \neq i}\rangle$:
$$\xymatrix{
M \bullet_\mu G \ar@{->>}[r] \ar@{->>}[d]_{\psi_i} & G \ar@{=}[d]\\
\langle \overline{\tau_i} \rangle \bullet_{\mu_i} G \ar@{->>}[r] & G.
}$$
If $\hat \sigma \in M \bullet_\mu G$ is a lift of $\sigma$, then $\psi_i(\hat \sigma) \in \Gal(L_i/F)$ is a lift for $\sigma$.  Because we know the group structure of $\Gal(L_i/K)$, we can say that
$$\psi_i(\hat \sigma)^{p^n} = \left\{ \begin{array}{ll} {c_1\Psi^{\ell(\alpha_1)-1}} \tau_1,&\mbox{ if }i =1\\0,&\mbox{ if }i>1\mbox{ and }\ell(\alpha_i)<p^n\\{c_i\Psi^{p^n-1}}\tau_i,&\mbox{ if }\ell(\alpha_i) = p^n.\end{array}\right.$$ for some $c_1 \in \F_p^\times$ and $c_i \in \F_p$.  Hence we have
$$\hat \sigma^{p^n} = {c_1\Psi^{\ell(\alpha_1)-1}}\tau_1+\sum_{\ell(\alpha_i)=p^n} {c_i\Psi^{p^n-1}} \tau_i,$$ and it follows that $\Gal(L/F) \simeq M \bullet_{\lambda(M)} G$, as desired.
\end{proof}

In the remaining sections of this paper, we glean as much from this parameterization as we can.  In the next section we focus on enumerating solutions to embedding problems, and in the final section we determine certain restrictions on the absolute Galois group of $F$ that this parameterization imposes.

\section{Counting solutions to one embedding problem within another}\label{sec:counting}

Our main theorem has the following obvious

\begin{corollary}
For a finite $\F_p[G]$-module $M$, the number of solutions to an embedding problem $\xymatrix{M \bullet_\mu G \ar@{->>}[r] & G}$ over $K/F$ is equal to the number of modules $U \subseteq J(K)$ satisfying $U \simeq M$ and $\lambda(U) = \mu$.
\end{corollary}

In this section we make this corollary effective, in the sense that we use linear algebra to give explicit formulae for calculating the number of such solutions. Since it requires no more work and ensures finiteness of the associated modules, we will actually answer this question in a slightly different setting: for a given $A \subseteq J(K)$ corresponding to a finite extension $E/F$ solving $\xymatrix{A \bullet_\lambda G \ar@{->>}[r]& G}$, we will count the number of fields $L$ within $E/K$ that solve the embedding problem $\xymatrix{M \bullet_\mu G \ar@{->>}[r]& G}$ over $K/F$.  In other words, we count the number of solutions to a given embedding problem inside another.  Instead of simply looking for submodules $U \subseteq J(K)$ with $U \simeq M$ and $\lambda(U) = \mu$, in this case we need to satisfy the additional condition that $U \subseteq A$.  Though we don't always say so explicitly, in this section we always assume $A$ and $M$ are finite.%Though we may not always say so explicitly, the working assumption for the rest of this section is that $A$ and $M$ are both finite $\F_p[G]$-modules.

To begin, note that if $\mu < \lambda$ then no such submodule $M$ exists since $A$ contains no elements of length $\mu$ with non-trivial index. So we will only consider those cases where $\mu \geq \lambda$.  

Our strategy for counting the number of submodules of a particular type is to provide a recipe for building them from the ground up.  By this we mean that we provide a way for counting the total number of possibilities for the fixed part of such a submodule, then we count the number of ways to construct a module with the desired properties and a given fixed part.  

Before counting these submodules explicitly, we offer some preparatory lemmas.  In particular we will need to know the number of ways that one filtration of $\F_p$-spaces can be embedded into another.  Towards that end, suppose that $V$ and $W$ are finite $\F_p$-spaces, and that we have two filtrations
$F_V$ and $F_W$ of the form
\begin{equation}\label{eq:generic.filtration}
\begin{split}
F_V:\quad V &= V_1 \supseteq V_2 \supseteq \cdots \supseteq V_{p^n} \supseteq V_{p^n+1} = \{0\}\\
F_W:\quad W &= W_1 \supseteq W_2 \supseteq \cdots \supseteq W_{p^n} \supseteq W_{p^n+1} =\{0\}.
\end{split}\end{equation} We will say that $F_W \subseteq F_V$ if for each $i$ we have $W_i \subseteq V_i$, and that $F_{W'} \simeq F_W$ if $\dim(W_i')=\dim(W_i)$ for all $i$.  
In the case that $\dim(V)=n$ and $\dim(W)=m$, it is well known that the number of $\F_p$-subspaces isomorphic to $W$ within $V$ is given by
$$
\binom{n}{m}_p = 
\left\{\begin{array}{ll}
0, &\mbox{ if }m<0 \mbox{ or }m>n\\ 
\frac{(p^n-1)\cdots(p^{n-m+1}-1)}{(p^m-1)\cdots (p-1)},&\mbox{ if }0\leq m \leq n.\end{array}\right.
$$
The next lemma gives us the generalization of this result to the filtrations $F_V$ and $F_W$.

\begin{lemma}\label{le:flag.counter}
Suppose we have filtrations $F_V$ and $F_W$ as in Equation (\ref{eq:generic.filtration}). Then the number of filtrations isomorphic to $F_W$ within $F_V$ is $$\binom{F_V}{F_W} := \prod_{i=1}^{p^n} \binom{\dim(V_i) - \dim(W_{i+1})}{\dim(W_i)-\dim(W_{i+1})}_p.$$
\end{lemma}

\begin{proof}
First, observe that for an isomorphic copy of $W_i$ to be contained in $V_{i}$, we must have $\dim(W_i) \leq \dim(V_{i})$.  Hence $\dim(W_i) > \dim(V_{i})$ implies there is no filtration of subspaces of $V$ isomorphic to $F_W$.    In this case, we also have that the corresponding $p$-binomial coefficient $\left(\begin{array}{c}\dim(V_{i}) - \dim(W_{i+1})\\ \dim(W_i)- \dim(W_{i+1})\end{array}\right)_p = 0$, and so the identity holds.  

Suppose, then, that $\dim(W_i) \leq \dim(V_{i})$ for all $1\leq i\leq p^n$.  The number of $\F_p$-subspaces of $V_{p^n}$ isomorphic to $W_{p^n}$ is $\left(\begin{array}{c} \dim(V_{p^n}) \\ \dim(W_{p^n}) \end{array}\right)_p$.  Now suppose we have shown that the number of filtrations isomorphic to $W_\ell \supseteq \cdots \supseteq W_{p^n}$ within $V_\ell \supseteq \cdots \supseteq V_{p^n}$ is $$\prod_{i \geq \ell} \left(\begin{array}{c}\dim(V_{i}) -\dim(W_{i+1})\\\dim(W_i)-\dim(W_{i+1})\end{array}\right)_p.$$ Let $U_\ell \supseteq \cdots \supseteq U_{p^n}$ be one such filtration, and suppose we have $\F_p$-independent collections $\mathcal{I}_i$ such that $\cup_{i \geq k}\mathcal{I}_i$ is a basis for $U_k$ for each $k \geq \ell$.

The number of ways to choose a collection $\mathcal{I}_{\ell-1} \subseteq V_{\ell-1}$ of $\dim(W_{\ell-1})-\dim(W_\ell)$ elements so that $\cup_{i \geq \ell-1} \mathcal{I}_i$ is linearly independent is 
$$(p^{\dim(V_{\ell-1})}-p^{\dim(W_\ell)})(p^{\dim(V_{\ell-1})}-p^{\dim(W_\ell)+1})\cdots (p^{\dim(V_{\ell-1})}-p^{\dim(W_{\ell-1})-1}).$$   Repeating the same argument above, for a given $U_{\ell-1}$ there are $$(p^{\dim(W_{\ell-1})}-p^{\dim(W_\ell)})(p^{\dim(W_{\ell-1})}-p^{\dim(W_\ell)+1})\cdots (p^{\dim(W_{\ell-1})}-p^{\dim(W_{\ell-1})-1})$$ many choices for a set $\mathcal{I}_{\ell-1} \subseteq U_{\ell-1}$ which complete $\cup_{i \geq \ell} \mathcal{I}_i$ to a basis of $U_{\ell-1}$.  Hence the total number of ways to choose $U_{\ell-1}$ is 
\begin{equation*}
\begin{split}
&\frac {(p^{\dim(V_{\ell-1})}-p^{\dim(W_\ell)})}{(p^{\dim(W_{\ell-1})}-p^{\dim(W_\ell)})}
\frac{(p^{\dim(V_{\ell-1})}-p^{\dim(W_\ell)+1})}{(p^{\dim(W_{\ell-1})}-p^{\dim(W_\ell)+1})}
\cdots 
\frac{(p^{\dim(V_{\ell-1})}-p^{\dim(W_{\ell-1})-1})}{(p^{\dim(W_{\ell-1})}-p^{\dim(W_{\ell-1})-1})}
\\ &= 
\frac{(p^{\dim(V_{\ell-1})-\dim(W_\ell)}-1)}{(p^{\dim(W_{\ell-1})-\dim(W_\ell)}-1)}
\frac{(p^{\dim(V_{\ell-1})-\dim(W_\ell)-1}-1)}{(p^{\dim(W_{\ell-1})-\dim(W_\ell)-1}-1)}
\cdots
\frac{(p^{\dim(V_{\ell-1})-\dim(W_{\ell-1})+1}-1)}{(p-1)}
\\&= \left(\begin{array}{c}\dim(V_{\ell-1})-\dim(W_\ell)\\ \dim(W_{\ell-1})-\dim(W_\ell)\end{array}\right)_p.
\end{split}
\end{equation*}
\end{proof}

For $1 \leq \ell \leq p^n+1$, recall that we define $M_{\{\ell\}} = \im\left(\xymatrix{M \ar[r]^-{\Psi^{\ell-1}}& M}\right) \cap M^{G}$, and that these spaces give rise to the length filtration on $M$: $$F^{\tiny{\mbox{len}}}_M:\quad M^{G} = M_{\{1\}} \supseteq M_{\{2\}} \supseteq \cdots \supseteq M_{\{p^n\}} \supseteq M_{\{p^n+1\}} = \{1\}.$$  This filtration is essential for identifying and counting the appearance of modules isomorphic to $M$ within $A$, as shown in the next few lemmas.  %We leave the proof to the reader, as it is not difficult.

\begin{lemma}\label{le:identifying.modules.by.bottom}
Suppose $A \subseteq J(K)$ and that $F_W$ is a filtration $$W=W_1 \supseteq W_2 \supseteq \cdots \supseteq W_{p^n} \supseteq W_{p^n+1} = \{0\}.$$ Then $F_W \subseteq F^{\tiny{\mbox{len}}}_A$ and $F_W \simeq F^{\tiny{\mbox{len}}}_M$ if and only if there exists $U \subseteq A$ satisfying $U \simeq M$ and $F_W=F^{\tiny{\mbox{len}}}_U$.  
\end{lemma}
%$U \subseteq V$ satisfies $U \simeq M$ if and only if $U_{\{i\}} \subseteq V_{\{i\}}$ and $\dim(U_{\{i\}}) = \dim(M_{\{i\}})$ for all $1 \leq i \leq p^n$.

\begin{proof}
One direction is easy.  For the other, choose collections $\mathcal{I}_\ell$ such that $\cup_{j \geq \ell} \mathcal{I}_j$ is a basis for $W_\ell$, and let $\mathcal{I} = \cup_\ell \mathcal{I}_\ell$.  For each $x \in \mathcal{I}_\ell$, select an element $\alpha_x \in A$ such that $x = \Psi^{\ell-1} \alpha_x$; such a selection is possible since $x \in A_{\{\ell\}}$. Proposition \ref{prop:build.a.module} then tells us that $\oplus_{x \in \mathcal{I}} \langle \alpha_x \rangle \simeq M$.  
\end{proof}

The proof of the previous lemma tells us how to construct a module $U \simeq M$ with $F_W = F^{\mbox{\tiny{len}}}_U$.  The next lemma tells us how to determine when  modules constructed in this way are identical.

\begin{lemma}\label{le:module.differentiator}
Suppose $\mathcal{I} \subseteq A^{G}$ is a collection of $\F_p$-independent elements. For each $x \in \mathcal{I}$, let $1 \leq \ell_x \leq p^n$ be given so that $x \in A_{\{\ell_x\}}$, and suppose that $\alpha_x, \beta_x \in A$ satisfy $${\Psi^{\ell_x-1}}\alpha_x = x = {\Psi^{\ell_x-1}}\beta_x.$$  Then $\bigoplus_{x \in \mathcal{I}} \langle \alpha_x \rangle = \bigoplus_{x \in \mathcal{I}} \langle \beta_x \rangle$ if and only if for every $y \in \mathcal{I}$ we have $\beta_y \in \bigoplus_{x \in \mathcal{I}} \langle \alpha_x \rangle.$
\end{lemma}

\begin{proof}
Certainly one direction is trivial.  For the other, we show that $\oplus \langle \alpha_x \rangle \subseteq \oplus \langle \beta_x \rangle$ by induction on length.  The elements of length $1$ in both modules are simply the $\F_p$-span of the collection $\mathcal{I}$.  Suppose we know that any element of length $\ell-1$ within $\oplus \langle \alpha_x \rangle$ is contained in $\oplus \langle \beta_x \rangle$, and let $\alpha \in \oplus \langle \alpha_i \rangle$ be given so that $\ell(\alpha) = \ell$.  Hence 
$${\Psi^{\ell-1}}\alpha = \sum_{\ell_x \geq \ell} {c_x}x ={\Psi^{\ell-1}} \left(\sum_{\ell_x \geq \ell} {c_x\Psi^{\ell_x-\ell}}\beta_x\right).$$ It follows that $\hat\alpha:=\alpha-\sum{c_x\Psi^{\ell_x-\ell}} \beta_x$ has length less than $\ell$, and since each $\beta_y \in \oplus \langle \alpha_x \rangle$ it also follows that $\hat\alpha \in \oplus \langle \alpha_x \rangle$.  By induction, $\hat\alpha \in \oplus \langle \beta_x \rangle$, and hence so too $\alpha \in \oplus \langle \beta_x \rangle$.
\end{proof}

We now give a result that uses the length filtration to count the number of submodules of $A$ that are isomorphic to $M$.  For this and several subsequent results, we will adopt the following notation for a pair of filtrations $F_W$ and $F_V$ as in (\ref{eq:generic.filtration}):
$$\Omega(F_W,F_V) := \prod_{\ell=1}^{p^n} \left(p^{\sum_{j<\ell} \dim\left(V_j\right)-\dim\left(W_j\right)}\right)^{\codim\left(W_{\ell+1},W_\ell\right)}$$

\begin{theorem}\label{th:counting.M.inside.A}
The number of $U \subseteq A$ satisfying $U \simeq M$ is
$$\binom{F^{\mbox{\tiny{len}}}_A}{F^{\mbox{\tiny{len}}}_M}\Omega(F^{\mbox{\tiny{len}}}_M,F^{\mbox{\tiny{len}}}_A).$$
\end{theorem}

\begin{proof}
Lemma \ref{le:identifying.modules.by.bottom} characterizes the filtrations that are the fixed part of a submodule $U \subseteq A$ with $U \simeq M$.  By Lemma \ref{le:flag.counter}, the number of such filtrations is $\binom{F^{\mbox{\tiny{len}}}_A}{F^{\mbox{\tiny{len}}}_M}$.

Now suppose that $F_W$ is a filtration that satisfies the conditions of Lemma \ref{le:identifying.modules.by.bottom}, and we will count the number of $U\subseteq A$ such that $F^{\mbox{\tiny{len}}}_U = F_W$.  As in the proof of Lemma \ref{le:identifying.modules.by.bottom}, let $\{\mathcal{I}_\ell\}$ be chosen so that $\cup_{i \geq \ell} \mathcal{I}_i$ is a basis for $W_\ell$, and let $\mathcal{I} = \cup_\ell \mathcal{I}_\ell$.  We start by counting the number of sets $\{\alpha_x\}_{x \in \mathcal{I}}$ such that for any $x \in \mathcal{I}_\ell$ we have $x = \Psi^{\ell-1} \alpha_x$.  

Suppose that $\{\alpha_x\}$ is one such set, and let $\{\hat \alpha_x\}$ be another.  Then for each $x \in \mathcal{I}_\ell$ we have $\hat \alpha_x = g_x + \alpha_x$ for some $g_x \in A$ with $\ell(g_x)<\ell$.  Conversely, any choice of $\{g_x\} \subseteq A$ satisfying $\ell(g_x)<\ell$ for all $x \in \mathcal{I}_\ell$ gives rise to a set $\{\hat\alpha_x\}$ satisfying the necessary conditions.  The total number of choices for a given $g_x$ is the number of elements in $A$ with length less than $\ell$, which is $p^{\sum_{j<\ell}\dim\left(A_{\{j\}}\right)}.$  Hence the number of sets $\{\alpha_x\}$ is 
$$\prod_{\ell=1}^{p^n} \left(p^{\sum_{j<\ell} \dim\left(A_{\{j\}}\right)}\right)^{\codim\left(M_{\{\ell+1\}},M_{\{\ell\}}\right)}.$$

On the other hand, by Lemma \ref{le:module.differentiator} two sets $\{\alpha_x\}$ and $\{\hat \alpha_x\}$ satisfy $\oplus\langle \alpha_x \rangle = \oplus \langle \hat \alpha_x \rangle$ if and only if for each $y \in \mathcal{I}$ we have $g_y \in \oplus \langle \alpha_x \rangle$.  This occurs if and only if for each $y \in \mathcal{I}_\ell$ we have that $g_y$ is an element of $\oplus \langle \alpha_x \rangle \simeq M$ of length less than $\ell$, of which there are $p^{\sum_{j<\ell} \dim\left(M_{\{j\}}\right)}$ many choices.  Hence we have overcounted by a factor of 
$$\prod_{\ell=1}^{p^n} \left(p^{\sum_{j<\ell} \dim\left(M_{\{j\}}\right)}\right)^{\codim\left(M_{\{\ell+1\}},M_{\{\ell\}}\right)}.$$  Therefore the total number of such modules is the quotient of these two quantities, which is $\Omega\left(F^{\mbox{\tiny{len}}}_A,F^{\mbox{\tiny{len}}}_M\right)$.
\end{proof}

\begin{corollary}\label{cor:count.for.split.inside.split.case}
Suppose that $E/F$ corresponds to $A \subseteq J(K)$ with $\lambda:=\lambda(A)=p^n$. Then the number of solutions to the embedding problem $\xymatrix{M \bullet_{p^n} G \ar@{->>}[r] &G}$ over $K/F$ and within $E/F$ is $$\binom{F^{\mbox{\tiny{len}}}_A}{F^{\mbox{\tiny{len}}}_M} \Omega\left(F^{\mbox{\tiny{len}}}_M,F^{\mbox{\tiny{len}}}_A\right).$$
\end{corollary}

\begin{proof}
We know that solutions to the embedding problem $\xymatrix{M \bullet_{p^n} G \ar@{->>}[r] & G}$ over $K/F$ and within $E/F$ correspond to submodules $U \subseteq A$ with $U \simeq M$ and $\lambda(U) = p^n$.  Note that since $U \subseteq A$ and $\lambda(A)=p^n$, the condition $\lambda(U)=p^n$ is automatic.  Hence solutions to the relevant embedding problem correspond to submodules $U \subseteq A$ with $U \simeq M$.
\end{proof}

\begin{example}\label{ex:count.when.mu.equals.lambda.equals.p.to.the.n}
Here we do a specific example to illustrate the numerics.  Suppose that $p=5$ and $n=2$. Let $M \simeq \bigoplus_{\ell=1}^{25} \oplus_{d_\ell} \F_p[G]/\langle \Psi^\ell \rangle$ where $d_1=2, d_3=1$ and $d_{25}=3$, with all other $d_\ell=0$.  Suppose that $A \simeq \bigoplus_{\ell=1}^{25} \oplus_{e_\ell} \F_p[G]/\langle \Psi^\ell\rangle$ where $e_1 = 2$, $e_{24}=4$ and $e_{25}=5$, with all other $e_\ell=0$.  Hence the group $A \bullet_{p^n} G \simeq A \rtimes G$ has order $5^{225}$ and $M \bullet_{p^n} G \simeq M \rtimes G$ has order $5^{82}$.

We are interested in finding solutions to the embedding problem $\xymatrix{M \bullet_{p^n} G \ar@{->>}[r]& G}$ over $K/F$ that are within a given extension $E/F$ that solves the embedding problem $\xymatrix{A \bullet_{p^n} G \ar@{->>}[r] & G}$. By the previous result the essential quantities necessary to perform this calculation are the dimensions from the filtrations $F^{\mbox{\tiny{len}}}_M$ and $F^{\mbox{\tiny{len}}}_M$.  One has
$$\dim(M_{\{\ell\}}) = \left\{\begin{array}{ll}
3,&\mbox{ if }4 \leq \ell \leq 25\\
4,&\mbox{ if }2 \leq \ell \leq 3\\
6,&\mbox{ if }\ell=1\\
\end{array}\right\},
\quad 
\dim(A_{\{\ell\}}) = \left\{\begin{array}{ll}
5,&\mbox{ if }\ell=25\\
9,&\mbox{ if }2 \leq \ell \leq 24\\
11,&\mbox{ if }\ell=1\\
\end{array}\right\}.
$$
%and since $\lambda=1$ we also have $\dim(A^0_{\{1\}})=10$ by Lemma \ref{le:trivial.module.fixed.dims}. 

We first calculate $\binom{F^{\mbox{\tiny{len}}}_A}{F^{\mbox{\tiny{len}}}_M}$.  Since $\codim(M_{\{\ell\}},M_{\{\ell+1\}})=0$ for all $4 \leq \ell \leq 24$ and $\ell=2$, we only have to consider the terms coming from $\ell\in\{1,3,25\}$.   We therefore have $$\binom{F^{\mbox{\tiny{len}}}_A}{F^{\mbox{\tiny{len}}}_M} = 
%\binom{5}{3}_p \binom{8-3}{3-3}_p \binom{9-3}{3-3}_p \cdots \binom{9-3}{3-3}_p \binom{9-3}{4-3}_p \binom{9-4}{6-4}_p = 
\binom{5}{3}_5 \binom{9-3}{4-3}_5 \binom{11-4}{6-4}_5 = 1008467924179716.
$$
To calculate $\Omega\left(F^{\mbox{\tiny{len}}}_M,F^{\mbox{\tiny{len}}}_A\right)$, note that the only terms that are nontrivial correspond to $\ell\in \{3,25\}$: the exponent is trivial when $\ell=1$ and $\codim(M_{\{\ell\}},M_{\{\ell+1\}})=0$ for other $\ell$.  Hence we have
\begin{equation*}\begin{split}
\Omega\left(F^{\mbox{\tiny{len}}}_M,F^{\mbox{\tiny{len}}}_A\right)&=\left(p^{\sum_{j<3}\codim\left(M_{\{j\}},A_{\{j\}}\right)}\right)^{\codim\left(M_{\{4\}},M_{\{3\}}\right)} \left(p^{\sum_{j<25} \codim(M_{\{j\}},A_{\{j\}})}\right)^{\dim\left(M_{\{25\}}\right)} \\&= 5^{10}5^{3(141)} = 5^{433}.
\end{split}\end{equation*}
%Since $\codim\left(M_{\{25\}},M_{\{2\}}\right)=1$, 
So the total number of solutions to the embedding problem $M \bullet_{p^n} G \twoheadrightarrow G$ over $K/F$ within a given solution to $A \bullet_{p^n} G \twoheadrightarrow G$ over $K/F$ is
$$1008467924179716\cdot 5^{433} \approx 10^{317}.$$
\end{example}

Unfortunately these tools aren't quite sufficient for enumerating solutions to an embedding problem $\xymatrix{M \bullet_\mu G \ar@{->>}[r] & G}$ within a solution to another embedding problem $\xymatrix{A\bullet_\lambda G \ar@{->>}[r]&G}$ if $\lambda<p^n$, since in this case we'll have to be careful that the modules we collect have the appropriate index.  Since controlling elements of nontrivial index is critical for our analysis, we will need to do more.  

\begin{definition}
For $A \subseteq J(K)$, we define
$$A^0 = \{\alpha \in A: \ell(\alpha) < p^n \mbox{ and } e(\alpha) = 0\} = \ker(N) \cap A.$$ 
We also define $A^0_{\{\ell\}}=\im\left(\xymatrix{A^0 \ar[r]^{\Psi^{\ell-1}}&A^0}\right) \cap (A^0)^G$; in particular, this means that $A^0_{\{p^n\}} = \{1\}$, since all elements of $A^0$ have length at most $p^n-1$.
\end{definition}

\begin{remark*}
Notice that $A^0_{\{\ell\}}$ is \emph{not} $\left(A_{\{\ell\}}\right)^0=A_{\{\ell\}} \cap \ker(e)$.  
\end{remark*}

\begin{lemma}\label{le:trivial.module.fixed.dims}
For $A \subseteq J(K)$, if $\ell<p^n$ then $\codim\left(A^0_{\{\ell\}},A_{\{\ell\}}\right) = \mathds{1}_{\ell=\lambda(A)}$.
\end{lemma}

\begin{proof}
Suppose first that $\lambda(A)=p^n$, and let $x \in A_{\{\ell\}}$ be given with $\ell<p^n$.  Any solution $\alpha \in A$ to $x ={\Psi^{\ell-1}} \alpha$ must have $e(\alpha) = 0$, since otherwise $\lambda(A) \leq \ell$.  Hence $x \in A^0_{\{\ell\}}$.  The same argument shows that if $\lambda(A)<p^n$ and $\ell<\lambda(A)$, then $A_{\{\ell\}} = A^0_{\{\ell\}}$.

Suppose, then, that $\lambda(A)<p^n$ and $\ell\geq \lambda(A)$. Let $\chi \in A$ be given with $\ell(\chi) = \lambda(A)$ and $e(\chi) = 1$; let $x \in A_{\{\ell\}}$ be given.  For any solution $\alpha \in A$ to $x ={\Psi^{\ell-1}} \alpha$ we have $e(\alpha  {-e(\alpha)}\chi) = 0$, and hence $$x  {-e(\alpha)\Psi^{\ell-1}}\chi = {\Psi^{\ell-1}}(\alpha {-e(\alpha)}\chi) \in A^0_{\{\ell\}}.$$  Now when $\ell>\lambda(A)=\ell(\chi)$ the left side of this equation is just $x$, and hence we have $x \in A^0_{\{\ell\}}$ as desired.  When $\ell=\lambda(A)$, this equation shows that $$A_{\{\lambda(A)\}} = \langle {\Psi^{\lambda(A)-1}}\chi \rangle \oplus A^0_{\{\lambda(A)\}}.$$  
\end{proof}

We now have the necessary machinery to count the number of solutions to the embedding problem $\xymatrix{M\bullet_\mu G \ar@{->>}[r]&G}$ over $K/F$ within a given solution to $\xymatrix{A\bullet_\lambda G \ar@{->>}[r]&G}$ over $K/F$.  As we have already seen, this amounts to counting those $U \subseteq A$ such that $U \simeq M$ and $\lambda(U) = \mu$.  Though we will handle separate cases depending on $\mu$, the strategy in each case follows that of Theorem \ref{th:counting.M.inside.A}: we determine the conditions that a filtration $F_W$ given by $$F_W:\quad W = W_1 \supseteq W_2 \supseteq \cdots \supseteq W_{p^n} \supseteq W_{p^n+1}=\{0\}$$ must satisfy to be the ``bottom" of such a module, and for each such filtration we count the number of $U \subseteq A$ with $U \simeq M$ and $\lambda(U) = \mu$ ``above" it (i.e., so that $F^{\mbox{\tiny{len}}}_U=F_W$).

In addition to $F^{\mbox{\tiny{len}}}_M$ and $F^{\mbox{\tiny{len}}}_A$, when $\lambda<p^n$ there is another filtration --- which we call $\widehat{F^{\mbox{\tiny{len}}}_A}$ --- that will  play an important role:
\begin{equation}\label{eq:filtrations.of.interest}
\begin{split}
%F_1&: \quad A^0_{\{1\}} \supseteq A^0_{\{2\}} \supseteq \cdots \supseteq A^0_{\{p^n-1\}} \supseteq A_{\{p^n\}} \supseteq A_{\{p^n+1\}} =  \{0\}\\
\widehat{F^{\mbox{\tiny{len}}}_A}&: \quad A^G=A_{\{1\}} \supseteq \cdots \supseteq A_{\{\lambda-1\}} \supseteq A^0_{\{\lambda\}} \supseteq A_{\{\lambda+1\}} \supseteq \cdots \supseteq A_{\{p^n\}}\supseteq A_{\{p^n+1\}} =   \{0\}.\\
%F_3&: \quad A^0_{\{1\}} \supseteq \cdots \supseteq A^0_{\{\lambda-1\}} \supseteq A_{\{\lambda\}} \supseteq A_{\{\lambda+1\}} \supseteq \cdots \supseteq A_{\{p^n\}}\supseteq A_{\{p^n+1\}} =  \{0\}\\
%F_4&: \quad A^0_{\{1\}} \supseteq \cdots \supseteq A^0_{\{\mu-1\}} \supseteq A_{\{\mu\}} \supseteq A_{\{\mu+1\}} \supseteq \cdots \supseteq A_{\{p^n\}}\supseteq A_{\{p^n+1\}} =  \{0\}\\
\end{split}\end{equation}
%For completeness, when $\lambda(A)=p^n$ we will define $\widehat{F^{\mbox{\tiny{len}}}_A} = F^{\mbox{\tiny{len}}}_A$.

\begin{lemma}\label{le:fixed.part.for.split}
Suppose that $A \subseteq J(K)$ with $\lambda:=\lambda(A)<p^n$  and that $F_W$ is a filtration $$F_W:\quad W = W_1 \supseteq W_2 \supseteq \cdots \supseteq W_{p^n} \supseteq W_{p^n+1} = \{0\}.$$ Then there exists $U\subseteq A$ with $U \simeq M$, $\lambda(U) = p^n$, and $F^{\tiny{\mbox{len}}}_U = F_W$ if and only if both of the following conditions hold:
\begin{equation}\begin{split}\label{eq:filtration.conditions.split.case}
 F_W \simeq F^{\tiny{\mbox{len}}}_M \quad \mbox{ and } \quad 
F_W \subseteq \widehat{F^{\mbox{\tiny{len}}}_A}.
\end{split}\end{equation}
\end{lemma}

\begin{proof}
First, suppose $U \subseteq A$ with $U \simeq M$ and $\lambda(U) = p^n$. The condition $U \simeq M$ implies $F^{\mbox{\tiny{len}}}_U \simeq F^{\tiny{\mbox{len}}}_M$. The condition $F^{\mbox{\tiny{len}}}_U \subseteq F^{\mbox{\tiny{len}}}_A$ comes from the inclusion $U \subseteq A$. The additional restriction that $\lambda(U) = p^n$ implies that $U_{\{\ell\}}\subseteq A^0_{\{\ell\}}$ for each $1 \leq \ell < p^n$; otherwise there would be some $1 \leq \ell < p^n$ and an element $x \in U_{\{\ell\}}$ such that any solution $u \in U$ to $x = {\Psi^{\ell-1}}u$ would satisfy $e(u) \neq 0$, contradicting $\lambda(U) = p^n$.  Now when $\lambda \neq \ell < p^n$ then $A^0_{\{\ell\}} = A_{\{\ell\}}$ by Lemma \ref{le:trivial.module.fixed.dims}.  Hence we have $F^{\mbox{\tiny{len}}}_U \subseteq \widehat{F^{\mbox{\tiny{len}}}_A}$.

Conversely, suppose $F_W$ satisfies (\ref{eq:filtration.conditions.split.case}).  Choose a basis $\mathcal{I}_{p^n}$ for $W_{p^n}$, and for each $\ell<p^n$ select a basis $\mathcal{I}_\ell$ for a complement of $W_{\ell+1}$ within $W_\ell$.  For each $x \in \mathcal{I}_\ell$ choose an element $\alpha_x$ so that $x = {\Psi^{\ell-1}}\alpha_x$. When $\ell<p^n$ we can --- and do --- choose $\alpha_x$ such that $e(\alpha_x) = 0$; this is possible since $x \in A^0_{\{\ell\}}$ when $\ell = \lambda$, and $A^0_{\{\ell\}} = A_{\{\ell\}}$ when $\lambda \neq \ell < p^n$.  Then $U = \oplus_{x \in \mathcal{I}} \langle \alpha_x \rangle$ satisfies $U \simeq M$, $\lambda(U) = p^n$ and $F^{\mbox{\tiny{len}}}_U = F_W$.
\end{proof}

\begin{theorem}\label{th:count.for.split.group}
Suppose that $E/F$ corresponds to $A \subseteq J(K)$ with $\lambda:=\lambda(A)<p^n$.  Then the number of solutions to the embedding problem $\xymatrix{M \bullet_{p^n} G \ar@{->>}[r]& G}$ over $K/F$ and within $E/F$ is 
$$
%\prod_{i=1}^{p^n} \left(\begin{array}{c}\Delta(A_{\{i\}}) - \Delta(M_{\{i+1\}}) - \mathds{1}_{i=\lambda}\cdot \mathds{1}_{i \neq p^n}\\\Delta\left(M_{\{i\}}\right)-\Delta\left(M_{\{i+1\}}\right)\end{array}\right)_p 
\left(\begin{array}{c}\widehat{F^{\mbox{\tiny{len}}}_A}\\F^{\mbox{\tiny{len}}}_M\end{array}\right)
%\prod_{i=1}^{p^n}\left(p^{\sum_{j<i} \Delta(A_{\{j\}})-\Delta(M_{\{j\}})-\mathds{1}_{j=\lambda}\cdot \mathds{1}_{i\neq p^n}}\right)^{\Delta\left(M_{\{i\}}\right)-\Delta\left(M_{\{i+1\}}\right)}
%\left(\prod_{i=1}^{p^n-1}\left(p^{\sum_{j<i} \codim(M_{\{j\}},A^0_{\{j\}})}\right)^{\codim(M_{\{i+1\}},M_{\{i\}})}\right)\left(p^{\sum_{j<p^n} \codim(M_{\{j\}},A_{\{j\}})}\right)^{\dim\left(M_{\{p^n\}}\right)}
\Omega\left(F^{\mbox{\tiny{len}}}_M,F^{\mbox{\tiny{len}}}_A\right)\left(p^{-\codim\left(M_{\{p^n\}},M_{\{\lambda+1\}}\right)}\right)
.$$
\end{theorem}

\begin{proof}
By Theorem \ref{th:main.theorem} we know that a solution to this embedding problem within $E/F$ corresponds to a submodule 
\begin{equation}\label{eq:U.conditions.split.case}\mbox{$U \subseteq A$ so that $U \simeq M$ and $\lambda(U) = p^n$.}\end{equation}  %To count all such submodules, we'll count the number of filtrations $W = W_1 \supseteq W_2 \supseteq \cdots \supseteq W_{p^n} \supseteq W_{p^n+1} = \{1\}$ within $A^G$ for which there exists a submodule $U \subseteq A$ with $U\simeq M$, $\lambda(U) = 0$ and $U_{\{i\}} = W_i$.  For each such filtration, we will then count the number of modules $U$ as above.
Lemma \ref{le:fixed.part.for.split} characterizes the filtrations $F_W$ that are the fixed part of such a submodule $U$. By Lemma \ref{le:flag.counter}, the number of such filtrations is $\binom{\widehat{F^{\mbox{\tiny{len}}}_A}}{F^{\mbox{\tiny{len}}}_M}$.

%Now we claim that $$A_{\{i\}} = \left\{\begin{array}{ll}K_{\{i\}},&\mbox{ if }i \neq \lambda\\K_{\{i\}} \oplus \langle \delta^{(\sigma-1)^{\lambda - 1}} \rangle,&\mbox{ if }i = \lambda \mbox{ and }\delta \in A \mbox{ such that }e(\delta) \neq 0, \ell(\delta) = \lambda. \end{array}\right.$$ Once this claim is established, we will have accounted for the first factor in the product that appears in the conclusion of this theorem.  The only non-obvious inclusion is to show that $A_{\{i\}}$ is in the stated subspace, so let $\beta \in A_{\{i\}}$.  Suppose first that $i<\lambda$. Then if $\beta = \alpha^{(\sigma-1)^{i-1}}$ we must have $e(\alpha) = 0$ since $\lambda(A) = \lambda$ and $\ell(\alpha) = i < \lambda$.  So suppose instead that $i \geq \lambda$ and  that $\beta = \alpha^{\Psi^{i-1}}$ with $e(\alpha) = c \neq 0$.  If we select $\delta \in A$ so that $\ell(\delta) = \lambda$ and $e(\delta) = 1$, then the element $\hat \alpha = \alpha/\delta^c$ has $e(\hat \alpha) = 0$ and $$\beta = \hat \alpha^{\Psi^{i-1}}  (\delta^c)^{\Psi^{i-1}} \in K_{\{i\}} \oplus \langle \delta^{\Psi^{\lambda - 1}} \rangle.$$  In particular, when $\lambda < i$ this expression simplifies to $\beta =  \hat \alpha^{\Psi^{i-1}} \in K_{\{i\}}$.

So suppose we have chosen a filtration $F_W$ satisfying (\ref{eq:filtration.conditions.split.case}), and let $\mathcal{I}$ be a basis for $W$ as in the proof of Lemma \ref{le:fixed.part.for.split}.  A submodule $U \subseteq A$ satisfies $W = U^{G}$ and (\ref{eq:U.conditions.split.case}) if and only if $U = \oplus_{x \in \mathcal{I}} \langle \alpha_x \rangle$ for elements $\{\alpha_x\} \subseteq A$ satisfying
\begin{equation}\label{eq:alpha.conditions.split.case}
\begin{array}{ll}
{\Psi^{\ell-1}}\alpha_x = x,&\mbox{ for all }x \in \mathcal{I}_\ell\\
e(\alpha_x) = 0, &\mbox{ for all }x \in \mathcal{I}_\ell \mbox{ with }\ell<p^n.
\end{array}
\end{equation}
We need to count the number of choices of $\{\alpha_x\}$ satisfying (\ref{eq:alpha.conditions.split.case}) which yield distinct modules.  

Following an argument similar to that in the proof of Theorem \ref{th:counting.M.inside.A}, the sets $\{\alpha_x\}$ satisfying (\ref{eq:alpha.conditions.split.case}) are in bijection with sets $\{g_x\}$ satisfying $\ell(g_x)<\ell$ for each $x \in \mathcal{I}_\ell$ and $e(g_x) = 0$ when $\ell<p^n$.  For each given $x \in B_\ell$, the total number of choices for a given $g_x$ is simply the number of elements of length less than $\ell$ within $A$ (and, when $\ell<p^n$, contained within $A^0$):$$\left\{\begin{array}{ll}p^{\sum_{j<\ell} \dim\left(A^0_{\{j\}}\right)} = p^{\sum_{j<\ell}\dim\left(A_{\{j\}}\right)}p^{-\mathds{1}_{\ell>\lambda}},&\mbox{ if }\ell<p^n\\
p^{\sum_{j<\ell} \dim(A_{\{j\}})},&\mbox{ if } \ell=p^n.\end{array}\right.$$  Hence the total number of choices for the collection $\{\alpha_x\}$ satisfying (\ref{eq:alpha.conditions.split.case}) is given by 
\begin{equation}\begin{split}\label{eq:counting.alpha.sets.mu.equals.p.to.the.n}
\prod_{\ell=1}^{p^n} &\left( p^{\sum_{j<\ell} \dim\left(A_{\{j\}}\right)}p^{-\mathds{1}_{\lambda<\ell<p^n}}\right)^{\codim\left(M_{\{\ell+1\}},M_{\{\ell\}}\right)} \\&=\left(\prod_{\ell=1}^{p^n} \left(p^{\sum_{j<\ell} \dim\left(A_{\{j\}}\right)}\right)^{\codim\left(M_{\{\ell+1\}},M_{\{\ell\}}\right)}\right)~p^{-\codim\left(M_{\{p^n\}},M_{\{\lambda+1\}}\right)}.
\end{split}
\end{equation}

Again following the proof of Theorem \ref{th:counting.M.inside.A}, we can use Lemma \ref{le:module.differentiator} to show that this overcounts by a factor of $$\prod_{\ell=1}^{p^n} \left(p^{\sum_{j<\ell}\dim(M_{\{j\}})}\right)^{\codim\left(M_{\{\ell+1\}},M_{\{\ell\}}\right)}.$$  The quotient of these two quantities gives the desired result.

%To reach the desired conclusion, we make the relevant substitutions from Lemma \ref{le:trivial.module.fixed.dims}.
\end{proof}

\begin{remark*}
One can choose to use $\Omega\left(F^{\mbox{\tiny{len}}}_M,\widehat{F^{\mbox{\tiny{len}}}_A}\right)$ in place of $\Omega\left(F^{\mbox{\tiny{len}}}_M,F^{\mbox{\tiny{len}}}_A\right)$ in the previous theorem.  If one does this, then the number of solutions to the embedding problem is
$$\left(\begin{array}{c}\widehat{F^{\mbox{\tiny{len}}}_A}\\F^{\mbox{\tiny{len}}}_M\end{array}\right)
\Omega\left(F^{\mbox{\tiny{len}}}_M,\widehat{F^{\mbox{\tiny{len}}}_A}\right)\left(p^{\dim\left(M_{\{p^n\}}\right)}\right).
$$ To see this, one substitutes $\dim\left(A^0_{\{\lambda\}}\right)$ in place of $\dim\left(A_{\{\lambda\}}\right)$ in (\ref{eq:counting.alpha.sets.mu.equals.p.to.the.n}); doing so kills the factor of $p^{-1}$ that appeared for in the terms corresponding to  $\lambda<\ell<p^n$.  On the other hand,  the term corresponding to $\ell=p^n$ uses $\dim\left(A_{\{\lambda\}}\right)=\dim\left(A^0_{\{\lambda\}}\right)+1$ in its evaluation, and so we pick up an additional factor of $p^{\dim\left(M_{\{p^n\}}\right)}$.
\end{remark*}

\begin{example}\label{ex:count.when.mu.equals.p.to.the.n}
We will again consider the modules $M$ and $A$ from Example \ref{ex:count.when.mu.equals.lambda.equals.p.to.the.n}, though this time we'll assume $\lambda :=\lambda(A)=1$.  Hence we are interested in counting solutions to the embedding problem $\xymatrix{M \bullet_{p^n} G \ar@{->>}[r]& G}$ over $K/F$ that are within a given extension $E/F$ that solves the embedding problem $\xymatrix{A \bullet_1 G \ar@{->>}[r] & G}$. 

We first calculate $\binom{\widehat{F^{\mbox{\tiny{len}}}_A}}{F^{\mbox{\tiny{len}}}_M}$.  As before we only have to consider the terms coming from $\ell\in\{1,3,25\}$, since $\codim\left(M_{\{\ell+1\}},M_{\{\ell\}}\right)=0$ otherwise.   We have $$\binom{\widehat{F^{\mbox{\tiny{len}}}_A}}{F^{\mbox{\tiny{len}}}_M} = 
%\binom{5}{3}_p \binom{8-3}{3-3}_p \binom{9-3}{3-3}_p \cdots \binom{9-3}{3-3}_p \binom{9-3}{4-3}_p \binom{9-4}{6-4}_p = 
\binom{5}{3}_5 \binom{9-3}{4-3}_5 \binom{10-4}{6-4}_5 = 40326324754716.
$$
Since $\codim\left(M_{\{25\}},M_{\{2\}}\right)=1$, we can use our previous calculation of $\Omega\left(F^{\mbox{\tiny{len}}}_M,F^{\mbox{\tiny{len}}}_A\right) = 5^{433}$ to give the total number of solutions to the embedding problem $M \bullet_{p^n} G \twoheadrightarrow G$ over $K/F$ within a given solution to $A \bullet_1 G \twoheadrightarrow G$ over $K/F$ as
$$40326324754716\cdot 5^{433} \cdot 5^{-1} \approx 10^{315}.$$
\end{example}

We now proceed to give the generic formula for the number of  solutions to $\xymatrix{M \bullet_\mu G \ar@{->>}[r] & G}$ within a given solution to $\xymatrix{A \bullet_\lambda G \ar@{->>}[r] & G}$ when  $\lambda = \mu < p^n$.  

\begin{lemma}\label{le:filtration.conditions.lambda.equals.mu}
Suppose that $A\subseteq J(K)$ with $\lambda:=\lambda(A) < p^n$ and that  $F_W$ is a filtration $$W = W_1 \supseteq W_2 \supseteq \cdots \supseteq W_{p^n} \supseteq W_{p^n+1} = \{0\}.$$ Then there exists $U \subseteq A$ with $U \simeq M$, $\lambda(U) = \lambda$ and $F^{\mbox{\tiny{len}}}_U = F_W$ if and only if all the following conditions hold:
\begin{equation}\label{eq:filtration.conditions.lambda.case}\begin{split}
F_W \simeq F^{\mbox{\tiny{len}}}_M, \quad F_W \subseteq F^{\mbox{\tiny{len}}}_A,\quad \mbox{ and } \quad F_W \not\subseteq \widehat{F^{\mbox{\tiny{len}}}_A}.
\end{split}\end{equation} 
\end{lemma}

\begin{remark*}
The final two conditions imply that $W_\lambda \cap \left(A_{\{\lambda\}} \setminus A^0_{\{\lambda\}}\right) \neq \emptyset$.
\end{remark*}
%gaussian_binomial(5,3,5)*gaussian_binomial(6,1,5)*gaussian_binomial(6,2,5)
%log(gaussian_binomial(5,3,5)*gaussian_binomial(6,1,5)*gaussian_binomial(6,2,5)*5.0^(432),10.)

\begin{proof}
Suppose that $U \subseteq A$ satisfies $U \simeq M$ and $\lambda(U) = \lambda$, and consider $F^{\mbox{\tiny{len}}}_U$.  The condition $U \simeq M$ implies $F^{\mbox{\tiny{len}}}_U \simeq F^{\mbox{\tiny{len}}}_M$.  The inclusion $U \subseteq A$ implies that $F^{\mbox{\tiny{len}}}_U \subseteq F^{\mbox{\tiny{len}}}_A$.  For the final condition, suppose to the contrary that $U_{\{\lambda\}} \subseteq A^0_{\{\lambda\}}$, and  let $u \in U$ be given so that $e(u) \neq 0$ and $\ell(u) = \lambda$.  Now by assumption ${\Psi^{\lambda-1}} u = {\Psi^{\lambda-1}} v$ for some $v \in A^0$, and so $u-v$ has $\ell(u-v)<\lambda$ and $e(u-v) \neq 0$.  This contradicts the definition of $\lambda = \lambda(A)$, and so our assumption that $U_{\{\lambda\}} \subseteq A^0_{\{\lambda\}}$ is false.

Conversely, suppose that $F_W$ is a filtration that satisfies  (\ref{eq:filtration.conditions.lambda.case}).  Choose a basis $\mathcal{I}_{p^n}$ for $W_{p^n}$, and for each $\ell<p^n$ select a basis $\mathcal{I}_\ell$ for a complement of $W_{\ell+1}$ within $W_\ell$.  For each $x \in \mathcal{I}_\ell$ choose an element $\alpha_x$ so that $x = {\Psi^{\ell-1}}\alpha_x$; when $\ell<\lambda$ we  choose $\alpha_x$ such that $e(\alpha_x) = 0$ since $x \in A^0_{\{\ell\}}$, and when $\ell = \lambda$ we must have $e(\alpha_x) \neq 0$ for some $x$ since $\mathcal{I}_\lambda \not\subseteq A^0_{\{\lambda\}}$.  Then the module $U = \oplus_{x \in \mathcal{I}} \langle \alpha_x \rangle$ satisfies the conditions $U \simeq M$ and $\lambda(U) = \lambda$.
\end{proof}

\begin{theorem}\label{th:count.for.lambda.case}
Suppose $E/F$ corresponds to $A \subseteq J(K)$ with $\lambda:=\lambda(A)<p^n$.  Suppose further that $M$ contains a summand of dimension $\lambda$.  Then the number of solutions to the embedding problem $\xymatrix{M \bullet_\lambda G \ar@{->>}[r]& G}$ over $K/F$ and within $E/F$ is 
\begin{equation*}\begin{split}
%\left(\binom{\Delta(A_{\{\lambda\}})-\Delta(M_{\{\lambda+1\}})}{\Delta\left(M_{\{\lambda\}}\right)-\Delta\left(M_{\{\lambda+1\}}\right)}_p\right. & \left.- \binom{\Delta(A_{\{\lambda\}})-\Delta(M_{\{\lambda+1\}})-1}{\Delta\left(M_{\{\lambda\}}\right)-\Delta\left(M_{\{\lambda+1\}}\right)}_p\right)\prod_{i\neq \lambda} \binom{\Delta(A_{\{i\}})-\Delta(M_{\{i+1\}})}{\Delta\left(M_{\{i\}}\right)-\Delta\left(M_{\{i+1\}}\right)}_p
\left(\binom{F^{\mbox{\tiny{len}}}_A}{F^{\mbox{\tiny{len}}}_M}-\binom{\widehat{F^{\mbox{\tiny{len}}}_A}}{F^{\mbox{\tiny{len}}}_M}\right)
%& \prod_{i=1}^{p^n}\left(p^{\sum_{j<i}\dim\left(A_{\{j\}}\right)-\dim\left(M_{\{j\}}\right)}\right)^{\codim\left(M_{\{i+1\}},M_{\{i\}}\right)}.
\Omega\left(F^{\mbox{\tiny{len}}}_M,F^{\mbox{\tiny{len}}}_A\right).
\end{split}\end{equation*}
\end{theorem}

\begin{proof}
We must count the number of submodules 
\begin{equation}\label{eq:U.conditions.mu.equals.lambda}U \subseteq A\mbox{ such that }U \simeq M\mbox{ and }\lambda(U) = \lambda.\end{equation}  As before, we do this by first counting the number of filtrations $F_W$ that are the fixed part of such a submodule, and then for each such filtration we count the number of submodules $U \subseteq A$ ``above'' this filtration which satisfy (\ref{eq:U.conditions.mu.equals.lambda}).

By Lemma \ref{le:filtration.conditions.lambda.equals.mu}, a filtration $F_W$ equals $F^{\mbox{\tiny{len}}}_U$ for some $U$ satisfying (\ref{eq:U.conditions.mu.equals.lambda}) if and only if $F_W$ satisfies (\ref{eq:filtration.conditions.lambda.case}).  Lemma \ref{le:flag.counter} tells us the number of such filtrations is $\binom{F^{\mbox{\tiny{len}}}_A}{F^{\mbox{\tiny{len}}}_M} - \binom{\widehat{F^{\mbox{\tiny{len}}}_A}}{F^{\mbox{\tiny{len}}}_M}.$

%Since $A_{\{\ell\}} = K_{\{\ell\}}$ for all $\ell \neq \lambda$ and $\Delta(A_{\{\lambda\}}) = \Delta(K_{\{\lambda\}})+1$, the number of choices for a fixed submodule for $B$ is 
%\\&=\left(\binom{\Delta(A_{\{\lambda\}})-\Delta(M_{\{\lambda+1\}})}{\Delta\left(M_{\{\lambda\}}\right)-\Delta\left(M_{\{\lambda+1\}}\right)}_p - \binom{\Delta(A_{\{\lambda\}})-\Delta(M_{\{\lambda+1\}})-1}{\Delta\left(M_{\{\lambda\}}\right)-\Delta\left(M_{\{\lambda+1\}}\right)}_p\right) \prod_{i\neq \lambda} \binom{\Delta(A_{\{i\}})-\Delta(M_{\{i+1\}})}{\Delta\left(M_{\{i\}}\right)-\Delta\left(M_{\{i+1\}}\right)}_p.

So suppose we have chosen a filtration $F_W$ satisfying (\ref{eq:filtration.conditions.lambda.case}), and let $\mathcal{I}$ be a basis for $W$ as in the proof of Lemma \ref{le:filtration.conditions.lambda.equals.mu}.  We claim that for any set of elements $\{\alpha_x\} \subseteq A$ satisfying
\begin{equation}\label{eq:alpha.conditions.lambda.case}
\begin{array}{ll}
{\Psi^{\ell-1}} \alpha_x = x,&\mbox{ for all }x \in \mathcal{I}_\ell\\
\end{array}
\end{equation}
 must have $U=\oplus \langle \alpha_x \rangle$ satisfying (\ref{eq:U.conditions.mu.equals.lambda}).  Certainly $U \simeq M$, but we will also have the necessary index conditions because
$W_{\ell} \subseteq A^0_{\{\ell\}}$ for all $i<\lambda$ (since no elements of length less than $\lambda$ in $A$ have nontrivial index) and $W_{\lambda} \not\subseteq A^0_{\{\lambda\}}$ by construction.

Hence we can repeat the arguments found in Theorem \ref{th:counting.M.inside.A} to tell us that the number of sets $\{\alpha_x\}$ which lead to distinct modules isomorphic to $M$ and satisfying the necessary index conditions is  $\Omega\left(F^{\mbox{\tiny{len}}}_A,F^{\mbox{\tiny{len}}}_M\right)$.
\end{proof}

\begin{example}
We will continue with the modules $M$ and $A$ introduced in Example \ref{ex:count.when.mu.equals.lambda.equals.p.to.the.n}, though this time we find the number of module $U \subseteq A$ so that $U \simeq M$ and $\lambda(U) = 1 = \lambda(A)$.

In fact, in Examples \ref{ex:count.when.mu.equals.lambda.equals.p.to.the.n} and \ref{ex:count.when.mu.equals.p.to.the.n} we already calculated the necessary quantities for this enumeration.  The total number of solutions to the embedding problem $\xymatrix{M\bullet_1 G \ar@{->>}[r]&G}$ over $K/F$ within a solution to $\xymatrix{A\bullet_1 G \ar@{->>}[r]&G}$ over $K/F$ is 
$$\left(1008467924179716-40326324754716\right)\cdot 5^{433} \approx 10^{317}.$$
\end{example}

\begin{lemma}\label{le:filtration.conditions.lambda.less.than.mu}
Suppose that $A\subseteq J(K)$ with $\lambda:=\lambda(A)$, that $\lambda < \mu<p^n$, and that $F_W$ is a filtration $$W = W_1 \supseteq W_2 \supseteq \cdots \supseteq W_{p^n} \supseteq W_{p^n+1} = \{0\}.$$ Then there exists $U \subseteq A$ with $U \simeq M$, $\lambda(U) = \mu$ and $F^{\mbox{\tiny{len}}}_U = F_W$ if and only if the following conditions hold:
\begin{equation}\label{eq:filtration.conditions.mu.case}\begin{split}
F_W \simeq F^{\mbox{\tiny{len}}}_M \quad \mbox{ and } \quad F_W \subseteq \widehat{F^{\mbox{\tiny{len}}}_A}.\end{split}\end{equation} 
\end{lemma}

\begin{proof}
The proof is almost identical to that of Lemma \ref{le:fixed.part.for.split}.  The only difference comes when constructing a module $U$ ``above" a given filtration.  We select sets $\{\mathcal{I}_\ell\}$ as before, and choose $\{\alpha_x\}$ so that for all $x \in \mathcal{I}_\ell$ we have $x = \Psi^{\ell-1} \alpha_x$.  In this case, though, we must choose $\alpha_x$ subject to the following conditions: $e(\alpha_x)=0$ for all $x \in \mathcal{I}_\ell$ with $\ell<\mu$ (which is possible since $x \in A^0_{\{\ell\}}$ for all $\ell<\mu$ by construction); and $e(\alpha_x) \neq 0$ for some $x \in \mathcal{I}_\mu$ (which is possible since $\mu>\lambda$).

\end{proof}

\begin{theorem}\label{th:count.for.mu.case}
Suppose that $E/F$ corresponds to $A \subseteq J(K)$ with $\lambda:=\lambda(A)<p^n$.  Suppose further that $M$ contains a summand of dimension $\mu$ satisfying $\lambda<\mu<p^n$.  Then the number of solutions to the embedding problem $\xymatrix{M \bullet_\mu G \ar@{->>}[r]& G}$ over $K/F$ and within $E/F$ is 
\begin{equation*}\begin{split}
%&\prod_{i=1}^{p^n}  \binom{\Delta(A_{\{i\}}) - \Delta(M_{\{i+1\}}) - \mathds{1}_{i=\lambda}}{\Delta\left(M_{\{i\}}\right)-\Delta\left(M_{\{i+1\}}\right)}_p \times \\
\binom{\widehat{F^{\mbox{\tiny{len}}}_A}}{F^{\mbox{\tiny{len}}}_M}
&\Omega\left(F^{\mbox{\tiny{len}}}_M,F^{\mbox{\tiny{len}}}_A\right)
\left(p^{-\codim\left(M_{\{\mu\}},M_{\{\lambda+1\}}\right)}-p^{-\codim\left(M_{\{\mu+1\}},M_{\{\lambda+1\}}\right)}\right).
%%%%%%%%%%%%%
%p^{-\codim\left(M_{\{\mu\}},M_{\{\lambda+1\}}\right)}\left(1-p^{-\codim\left(M_{\{\mu+1\}},M_{\{\mu\}}\right)}\right).
%%%%%%%%%%%%%%
%&\left[\prod_{i=1}^{p^n}\left(\frac{p^{\sum_{j<i}\dim(A_{\{j\}})-\mathds{1}_{j=\lambda}\cdot \mathds{1}_{i<\mu}}}{p^{\sum_{j<i}\dim(M_{\{j\}})}}\right)^{\codim\left(M_{\{i+1\}},M_{\{i\}}\right)}-\prod_{i=1}^{p^n}\left(\frac{p^{\sum_{j<i}\dim(A_{\{j\}})-\mathds{1}_{j=\lambda}\cdot \mathds{1}_{i\leq\mu}}}{p^{\sum_{j<i}\dim(M_{\{j\}})}}\right)^{\codim\left(M_{\{i+1\}},M_{\{i\}}\right)}\right].
\end{split}\end{equation*}
\end{theorem}

\begin{proof}
This time we are interested in those submodules 
\begin{equation}\label{eq:U.condition.mu.case}
\mbox{$U \subseteq A$ such that $U \simeq M$ and $\lambda(U) = \mu$.}
\end{equation}
By Lemma \ref{le:filtration.conditions.lambda.less.than.mu}, a filtration $F_W$ equals $F^{\mbox{\tiny{len}}}_U$ for some $U$ satisfying (\ref{eq:U.condition.mu.case}) if and only if $F_W$ satisfies (\ref{eq:filtration.conditions.mu.case}).  By Lemma \ref{le:flag.counter}, the number of such filtrations is $
\binom{\widehat{F^{\mbox{\tiny{len}}}_A}}{F^{\mbox{\tiny{len}}}_M}
%\prod_{i=1}^{\mu-1} \binom{\Delta(A^0_{\{i\}}) - \Delta(M_{\{i+1\}})}{\Delta\left(M_{\{i\}}\right)-\Delta\left(M_{\{i+1\}}\right)}_p ~ \prod_{i=\mu}^{p^n} \binom{\Delta(A_{\{i\}}) - \Delta(M_{\{i+1\}})}{\Delta\left(M_{\{i\}}\right)-\Delta\left(M_{\{i+1\}}\right)}_p
.$

So suppose we have chosen a filtration $F_W$ that satisfies (\ref{eq:filtration.conditions.mu.case}), and let $\mathcal{I}$ be a basis for $W$ as in the proof of Lemma \ref{le:filtration.conditions.lambda.less.than.mu}.  A submodule $U \subseteq A$ has $F_W = F^{\mbox{\tiny{len}}}_U$ and (\ref{eq:U.condition.mu.case}) if and only if $U = \oplus_{x \in \mathcal{I}} \langle \alpha_x \rangle$ for elements $\{\alpha_x\} \subseteq A$ satisfying
\begin{equation}\label{eq:alpha.conditions.mu.case}
\begin{array}{ll}
{\Psi^{\ell-1}} \alpha_x = x,&\mbox{ for all }x \in \mathcal{I}_\ell\\
e(\alpha_x) = 0 &\mbox{ for all }x \in \mathcal{I}_\ell \mbox{ with }\ell<\mu\\
e(\alpha_x) \neq 0 &\mbox{ for some }x \in \mathcal{I}_\mu.
\end{array}
\end{equation}
%We need to count the number of different choices of $\{\alpha_x\}$ satisfying (\ref{eq:alpha.conditions.mu.case}) which yield distinct modules. 

To enumerate sets $\{\alpha_x\}$ satisfying (\ref{eq:alpha.conditions.mu.case}), start by choosing elements $\{\beta_x\}_{x\in\mathcal{I}}$ so that 
\begin{equation*}\begin{split}
{\Psi^{\ell-1}} \beta_x=x&\mbox{ for all }x \in \mathcal{I}_\ell\\
e(\beta_x) = 0 &\mbox{ for all }x \in \mathcal{I}_\ell \mbox{ with }\ell<p^n.
\end{split}\end{equation*}  If $\{\alpha_x\}$ satisfies (\ref{eq:alpha.conditions.mu.case}), then for each $x \in \mathcal{I}_\ell$ there exists $h_x \in A$ such that $\alpha_x = h_x + \beta_x$ and  $\ell(h_x)<\ell$.  Note also that $e(\alpha_x) = e(h_x)$ when $\ell<p^n$.  Hence to enumerate the number of choices of $\{\alpha_x\}$ satisfying (\ref{eq:alpha.conditions.mu.case}), we will count the total number of ways to choose $\{h_x\} \subseteq A$ so that $e(h_x) = 0$ for all $x \in \mathcal{I}_\ell$ with $\ell<\mu$ and subtract the total number of ways to choose $\{h_x\} \subseteq A$ so that $e(h_x) = 0$ for all $x \in \mathcal{I}_\ell$ with $\ell \leq \mu$.  To do this, note that the number of elements of length less than $\ell$ within $A$ is $p^{\sum_{j<\ell}\dim(A_{\{j\}})}$, and the number of elements of length less than $\ell$ within $A$ that have trivial index is $p^{\sum_{j<\ell}\dim(A^0_{\{j\}})}= p^{\sum_{j<\ell}\dim(A_{\{j\}})-\mathds{1}_{j=\lambda}}$.  Therefore the total number of choices for $\{h_x\}$ --- and hence the total number of collections $\{\alpha_x\}$ satisying (\ref{eq:alpha.conditions.mu.case}) --- is given by
\begin{equation*}\begin{split}
\prod_{\ell=1}^{p^n}&\left(p^{\sum_{j<\ell}\dim(A_{\{j\}})-\mathds{1}_{j=\lambda}\cdot \mathds{1}_{\ell<\mu}}\right)^{\codim\left(M_{\{\ell+1\}},M_{\{\ell\}}\right)}-\prod_{\ell=1}^{p^n}\left(p^{\sum_{j<\ell}\dim(A_{\{j\}})-\mathds{1}_{j=\lambda}\cdot \mathds{1}_{\ell\leq\mu}}\right)^{\codim\left(M_{\{\ell+1\}},M_{\{\ell\}}\right)}\\
&=\left(\prod_{\ell=1}^{p^n}\left(p^{\sum_{j<\ell}\dim\left(A_{\{j\}}\right)}\right)^{\codim\left(M_{\{\ell+1\}},M_{\{\ell\}}\right)}\right)\left(p^{-\codim\left(M_{\{\mu\}},M_{\{\lambda+1\}}\right)}-p^{-\codim\left(M_{\{\mu+1\}},M_{\{\lambda+1\}}\right)}\right).
\end{split}\end{equation*}

Now suppose that $\{\alpha_x\}$ and $\{\hat \alpha_x\}$ are two collections satisfying (\ref{eq:alpha.conditions.mu.case}); we examine when $\oplus \langle \alpha_x \rangle = \oplus \langle \hat \alpha_x\rangle$.  Note that for each $x \in \mathcal{I}_\ell$ there exists $g_x \in A$ so that $\hat\alpha_x = g_x + \alpha_x$, and that $\ell(g_x)<\ell$. Notice also that $e(g_x) = 0$ for all $x \in \mathcal{I}_\ell$ with $\ell<\mu$, and that $e(g_x + \alpha_x) \neq 0$ for some $x \in \mathcal{I}_\mu$.  On the other hand, Lemma \ref{le:module.differentiator} tells us that $\oplus \langle \alpha_x \rangle = \oplus \langle \hat \alpha_x \rangle$ if and only if $g_x \in \oplus \langle \alpha_x \rangle$ for all $x \in \mathcal{I}$.  Hence we must count the number of collections $\{g_x\} \subseteq \oplus \langle \alpha_x \rangle$ satisying 
\begin{equation}\begin{split}
\ell(g_x) < \ell, &\mbox{ for all }x \in \mathcal{I}_\ell\\
e(g_x) = 0, &\mbox{ for all }x \in \mathcal{I}_\ell \mbox{ where }\ell<\mu\\
e(g_x + \alpha_x) \neq 0 &\mbox{ for some }x \in \mathcal{I}_\mu.
\end{split}\end{equation}
Since $\lambda(\oplus \langle \alpha_x \rangle) = \mu$, if the first condition is satisfied then the second and third conditions are automatically satisfied.  Hence we must only count the number of collections $\{g_x\} \subseteq \oplus \langle \alpha_x \rangle$ satisfying $\ell(g_x)<\ell$ for each $x \in \mathcal{I}_\ell$.  Since the number of elements of length $\ell$ within $\oplus \langle \alpha_x \rangle \simeq M$ is $p^{\sum_{j<\ell}\dim(M_{\{j\}})}$, the number of such collections is
$$\prod_{\ell=1}^{p^n} \left(p^{\sum_{j<\ell}\dim(M_{\{j\}})}\right)^{\codim\left(M_{\{\ell+1\}},M_{\{\ell\}}\right)}.$$
\end{proof}

\begin{example}
We again revisit the modules $M$ and $A$ from Example \ref{ex:count.when.mu.equals.lambda.equals.p.to.the.n}, though this time we'll study embedding problems of the form $\xymatrix{M \bullet_{3} G \ar@{->>}[r] &G}$ within a solution to $\xymatrix{A \bullet_1 G \ar@{->>}[r] &G}$.  

We've already done most of the necessary calculations.  Since $\codim\left(M_{\{3\}},M_{\{2\}}\right)=0$ and $\codim\left(M_{\{4\}},M_{\{2\}}\right)=1$, the total number of solutions to the embedding problem $\xymatrix{M\bullet_3 G \ar@{->>}[r]&G}$ over $K/F$ and within a solution to $\xymatrix{A\bullet_1 G \ar@{->>}[r]&G}$ over $K/F$ is
$$40326324754716\cdot 5^{433} \cdot \left(5^{0}-5^{-1}\right)\approx 10^{316}.$$
\end{example}

Now that we have counted solutions to the embedding problem $\xymatrix{M \bullet_\mu G \ar@{->>}[r] & G}$ over a given $G$-extension within a solution to the embedding problem $\xymatrix{A\bullet_\lambda G \ar@{->>}[r]&G}$ over the same $G$-extension, we put these results to work by giving a lower bound on the number of such solutions.

\begin{corollary}\label{cor:bounding.realization.multiplicity.with.free.summands}
Suppose that $E/F$ corresponds to $A\subseteq J(K)$ with $\lambda:=\lambda(A)<p^n$.  Suppose further that $M \not\simeq A$ and $\xymatrix{M \bullet_\mu G \ar@{->>}[r]&G}$ has at least one solution over $K/F$ within $E/F$.  Then $\xymatrix{M \bullet_\mu G \ar@{->>}[r]&G}$ has a least $p^{\dim\left(M_{\{p^n\}}\right)}$ many solutions over $K/F$ within $E/F$.
\end{corollary}

\begin{proof}
Suppose first that the number of elements in $M$ of length at most $p^n-1$ is smaller than the number of elements in $A$ of length at most $p^n-1$.  We then have 
$$\left(p^{\sum_{j<p^n} \dim\left(A_{\{j\}}\right)-\dim\left(M_{\{j\}}\right)}\right)^{\dim\left(M_{\{p^n\}}\right)}\geq p^{\dim\left(M_{\{p^n\}}\right)}.$$
This term appears as a factor of the number of solutions to the embedding problem $\xymatrix{M \bullet_{\mu} G\ar@{->>}[r] &G}$ over $K/F$ within $E/F$ in all of Theorems \ref{th:count.for.split.group}, \ref{th:count.for.lambda.case} and \ref{th:count.for.mu.case} and Corollary \ref{cor:count.for.split.inside.split.case}, and hence we have at least $p^{\dim\left(M_{\{p^n\}}\right)}$ solutions to this embedding problem. (Note that the non-integer factors in Theorems \ref{th:count.for.split.group} and \ref{th:count.for.mu.case} are absorbed by the other factors of $\Omega\left(F^{\mbox{\tiny{len}}}_M,F^{\mbox{\tiny{len}}}_{A}\right)$.)  

%Suppose, then, that the number of elements in $M$ of length at most $p^n-1$ is equal to the number of elements in $A$ of length at most $p^n-1$.  We claim this can only happen when $\lambda(M) = \lambda$. To see this is true, not that if $\lambda = p^n$ then for $\xymatrix{M\bullet_\mu G \ar@{->>}[r]&G}$ to be solvable within $E/F$, we must have $\lambda(M)=p^n$ as well. Otherwise $\lambda<p^n$, and if $\lambda(M)>\lambda$ then $M$ wouldn't have the same number of elements of length $\lambda$ as $A$ (since $A$ contains an element of length $\lambda$ with nontrivial index that cannot be contained in $M$).  

Suppose, then, that the number of elements in $M$ of length at most $p^n-1$ is equal to the number of elements in $A$ of length at most $p^n-1$.  In all cases, one of the factors in the count on solutions to $\xymatrix{M \bullet_{\lambda} G \ar@{->>}[r] &G}$ over $K/F$ and within $E/F$ is
$$\binom{\dim\left(A_{\{p^n\}}\right)}{\dim\left(M_{\{p^n\}}\right)}_p.$$  Since $M \not\simeq A$ but $M$ and $A$ have the same number of elements of length at most $p^n-1$, it must be that $\dim(A_{\{p^n\}}) >\dim(M_{\{p^n\}}$.  Then we have
$$\binom{\dim\left(A_{\{p^n\}}\right)}{\dim\left(M_{\{p^n\}}\right)}_p \geq \binom{\dim\left(M_{\{p^n\}}\right)+1}{\dim\left(M_{\{p^n\}}\right)}_p = \frac{p^{\dim\left(M_{\{p^n\}}\right)+1}-1}{p-1} \geq p^{\dim\left(M_{\{p^n\}}\right)}.$$
\end{proof}

\section{Embedding problems over a given $K/F$}\label{sec:general.applications}

We have already seen that a solution to the embedding problem $\xymatrix{M \bullet_\mu G \ar@{->>}[r]& G}$ over $K/F$ corresponds to a submodule $U \subseteq J(K)$ with $U \simeq M$ and $\lambda(U) = \mu$.  If one knows the module structure of $J(K)$ and a method for computing $\lambda(J(K))$ for a given extension $K/F$ with $\Gal(K/F) \simeq \Z/p^n\Z$, then one knows everything about embedding problems over $K/F$ with elementary $p$-abelian kernel.  On the other hand, if one can make general statements about module structures of $J(K)$ and $\lambda(J(K))$ across all fields $K$, then one can make connections between \emph{a priori} unrelated embedding problems.  These goals will be the focus of this section.

We begin with a discussion of the module structure for $J(K)$.  The investigation into the module structure of $J(K)$ began with Fadeev and Borevi\v{c}'s computations of $J(K)$ when $K$ is a local field (see \cite{Bo,F}).  Min\'{a}\v{c} and Swallow were able to compute the module structure of $J(K)$ when $\Gal(K/F) \simeq \Z/p\Z$ and $\xi_p \in K$ in \cite{MS1}. In the case that $\ch{K} \neq p$ and $\Gal(K/F) \simeq \Z/p^n\Z$ with $n \geq 1$, the module structure for $J(K)$ was computed in \cite[Th.~2]{MSS1} and \cite[Th.~2]{MSSauto}.   

To state the decomposition, recall our convention that $G=\Gal(K/F)=\langle \sigma \rangle \simeq \Z/p^n\Z$. We write $K_i$ for the intermediate field of degree $p^i$ over $F$, and write $G_i = \Gal(K_i/F)$.  We also assign an invariant $i(K/F) \in \{-\infty,0,\cdots,n-1\}$ as in the paragraph preceding the statement of Theorem \ref{th:counting.to.split.embedding.problem} in section \ref{sec:introduction}.  When $\ch{K} \neq p$, we write $\hat K$ for the field $K(\xi_p)$, and likewise denote $\hat K_i = K_i(\xi_p)$.  A generator for $\Gal(\hat K/K)$ is denoted $\epsilon$, and $\epsilon(\xi_p) = \xi_p^t$.  For a submodule $A \subseteq J(K)$ we write $\left.A \right|_{\epsilon = t}$ for the $t$-eigenspace of $\epsilon$ within $A$.

\begin{proposition}\label{prop:module.structure.char.not.p} Suppose $\ch{K} \neq p$ and $\Gal(K/F) \simeq \Z/p^n\Z$; additionally, assume that if $p=2$ and $n=1$, then $i(K/F) = -\infty$. Then
$$J(K) = \langle \chi \rangle \oplus \bigoplus_{i=0}^n Y_i,$$
where 
\begin{itemize}
\item $\chi \in J(K)$ is an element of nontrivial index and length $p^{i(K/F)}+1$, and
\item $Y_i \simeq \oplus_{\mathfrak{d}_i} \F_p[G_i]$, and $Y_i \subseteq J(K)^0$ and $\mathfrak{d}_i = \codim_{\F_p}\left(\left.\frac{\left(N_{\hat K_{i+1}/\hat F}(\hat K_{i+1}^\times)\right) \hat K^{\times p}}{\hat K^{\times p}}\right|_{\epsilon=t},\left.\frac{\left(N_{\hat K_i/\hat F}(\hat K_i^\times)\right) \hat K^{\times p}}{\hat K^{\times p}}\right|_{\epsilon=t}\right).$
\end{itemize}
\end{proposition}

\begin{remark*}
When $p=2$, $n=1$ and $i(K/F) = 0$, the module structure of $J(K)$ is simply $J(K) \simeq Y_0 \oplus Y_1$, with the modules $Y_i$ satisfying the properties above (see \cite[Th.~1]{MSS1}).  In this case there aren't any elements of nontrivial index, whence the lack of a summand of the form $\langle \chi \rangle$ as in the other cases.  Note that this means $\lambda\left(J(K)\right) = p^{i(K/F)}+1 =p^n$ in this case.  In fact, it is not hard to see that $p^{i(K/F)}+1 = p^n$ if and only if $p=2$, $n=1$ and $i(K/F)=0$.
\end{remark*}

The module structure for $J(K)$ when $\ch{K} = p$ is very similar, if a little simpler.

\begin{proposition}\label{prop:module.structure.char.p}
Suppose $\ch{K}=p$, and $G = \Gal(K/F) \simeq \Z/p^n\Z$.  Then $i(K/F) = -\infty$, and
$$J(K) = \langle \chi \rangle \oplus Y_n$$
where
\begin{itemize}
\item $\chi \in J(K)$ is an element of nontrivial index and length $1$, and
\item $Y_n \simeq \oplus_{\mathfrak{d}_n} \F_p[G]$ and $\mathfrak{d}_n = \dim_{\F_p}\left(\frac{Tr_{K/F}(K) + \wp(K)}{\wp(K)}\right) = \dim_{\F_p}\left(\frac{F + \wp(K)}{\wp(K)}\right)$
\end{itemize}
\end{proposition}

\begin{proof}
The embedding problem
$$\xymatrix{\Z/p^{n+1}\Z \ar@{->>}[r] & \Z/p^n\Z}$$
is central and nonsplit, and hence has a solution over $K/F$ by \cite[App.~A]{JLY}. Let $\chi \in (K/\wp(K))^G$ be a class which generates such an extension, and for convenience let $e(\chi) = 1$.  We will show that for any $f \in (K/\wp(K))^{G}$ and any $1 \leq i \leq p^n$ there exists an element $k_i \in K/\wp(K)$ and a value $c \in \F_p$ such that $f = c\chi + \Psi^{i-1}k_i$; furthermore, when $i<p^n$ we may also insist that $e(k_i) = 0$.  Notice that since $\Psi^{p^n-1} \equiv \sum_{i=0}^{p^n-1} \sigma^i = Tr_{K/F}$, this result tells us that $f = c\chi + Tr_{K/F}(k_{p^n})$.  This gives the desired result.

First, let $f \in (K/\wp(K))^{G}$ be given.  If $e(f) = c$, then the element $k_1=f-c\chi$ satisfies the necessary conditions.  If $k_1=0$ then let $k_{p^n} = 0$; otherwise suppose $1\leq i<p^n$, and we have a nonzero element $k_i \in K/\wp(K)$ with $e(k_i)=0$ and $\Psi^{i-1}k_i = k_1 = f-c\chi$.  This means that we have a solution to the embedding problem 
$$\xymatrix{\F_p[G]/\langle\Psi^i\rangle \rtimes G  \ar@{->>}[r]& G}.$$  Notice that we have a short exact sequence
$$\xymatrix{1 \ar[r]& \F_p \simeq \langle \Psi^i \rtimes 1 \rangle \ar[r]& \F_p[G]/\langle\Psi^{i+1}\rangle \rtimes G \ar[r]& \F_p[G]/\langle\Psi^i\rangle \rtimes G \ar[r] & 1}.$$ Since the action of $\sigma$ is trivial on the kernel, this is a central extension.  Furthermore the sequence is nonsplit because $\F_p[G]/\langle\Psi^{i+1}\rangle$ and $\F_p[G]/\langle\Psi^{i}\rangle$ each have rank $2$.  Again applying \cite[App.~A]{JLY}, this embedding problem has a solution, and so there is an element $k_{i+1} \in K/\wp(K)$ with $\Psi k_{i+1} = k_i$ and either $\ell(k_{i+1}) = p^n$ or $e(k_{i+1})=0$.  By induction, the desired result follows.
\end{proof}

For any field $K$, these results show us that the possibilities for the $\F_p[G]$-module structure for $J(K)$ are quite limited, in the sense that most isomorphism classes of indecomposable $\F_p[G]$-modules cannot appear as summands of $J(K)$ for any $K$.  To make this more precise, we generalize some notation from the introduction.  
Recall that $\varepsilon(i)$ is defined by $p^{\varepsilon(\ell)-1}<\ell<p^{\varepsilon(\ell)}$, and define 
$$\mathfrak{D}_\ell = \left\{\begin{array}{ll}
\dim_{\F_p}\left(\left.\frac{\left(N_{\hat K_{\varepsilon(\ell)}/\hat F}\left(\hat K_{\varepsilon(\ell)}^\times\right)\right)\hat K^{\times p}}{\hat K^{\times p}}\right|_{\epsilon = t}\right),&\mbox{ if }\ch{K} \neq p\\[15pt]
\dim_{\F_p}\left(\frac{Tr_{K_{\varepsilon(\ell)}/F}\left(K_{\varepsilon(\ell)}\right)+\wp(K)}{\wp(K)}\right),&\mbox{ if }\ch{K}=p.\\
\end{array}\right.$$
The computed module structures for $J(K)$ allow us to calculate the essential quantities for understanding embedding problems over $K/F$.
\begin{corollary}\label{cor:computing.length.filtration.dimensions.for.JK}
If $\Gal(K/F) \simeq \Z/p^n\Z$, then $\lambda\left(J(K)\right) = p^{i(K/F)}+1$.  Morever, for all $1 \leq \ell \leq p^n$, \begin{equation*}
\dim\left({J(K)}_{\{\ell\}}\right) = \left\{\begin{array}{ll}
\mathfrak{D}_\ell +1,&\mbox{ if }\ell = p^{i(K/F)}+1 \mbox{ and either }p>2,n>1 \mbox{ or }i(K/F)=-\infty\\
\mathfrak{D}_\ell,&\mbox{ otherwise.}\\
\end{array}\right.
\end{equation*}
\end{corollary}
%It is worth observing that the quantities $\mathfrak{D}_i$ capture the dimensions of the subspaces from the filtration $\widehat{F^{\mbox{\tiny{len}}}_{J(K)}}$.

Before going into more technical results concerning general embedding problems, we are already in position to give a very quick proof of Theorem \ref{th:free.summands.realization.multiplicity}.

\begin{proof}[Proof of Theorem \ref{th:free.summands.realization.multiplicity}]
The hypotheses give $M \neq J(K)$.  Apply Corollary \ref{cor:bounding.realization.multiplicity.with.free.summands} with $A = J(K)$.
\end{proof}

We now give a characterization of fields $K/F$ which admit a solution to a given embedding problem $\xymatrix{M \bullet_\mu G \ar@{->>}[r]& G}$.  

\begin{theorem}\label{th:embedding.problem.criterion}
The embedding problem $\xymatrix{M \bullet_\mu G \ar@{->>}[r]& G}$ has a solution over $K/F$ if and only if the following conditions hold:
\begin{enumerate}
\item for all $1 \leq \ell \leq p^n$ we have $$\dim\left(M_{\{\ell\}}\right) \leq \left\{\begin{array}{ll}\mathfrak{D}_{\ell} + 1,&\mbox{ if }\ell=p^{i(K/F)}+1=\mu\mbox{ and either }p>2,n>1\mbox{ or }-\infty = i(K/F)\\
\mathfrak{D}_{\ell},&\mbox{ otherwise.}\end{array}\right.$$
%\item $\dim\left(M_{\{i\}}\right) \leq \mathfrak{D}_{i}$ if $i \neq p^{i(K/F)}+1$ or $p=2,n=1$ and $0=i(K/F)$; and
\item $\mu \geq p^{i(K/F)}+1$.  
\end{enumerate}
\end{theorem}

\begin{proof}
Since the smallest length for an element with non-trivial index in $J(K)$ is $p^{i(K/F)}+1$, any submodule $U \subseteq J(K)$ must have $\lambda(U) \geq p^{i(K/F)}+1$.  Hence if the second condition fails, then there is no solution to the corresponding embedding problem.

So suppose that $\mu \geq p^{i(K/F)}+1$.  Lemmas \ref{le:identifying.modules.by.bottom}, \ref{le:fixed.part.for.split}, \ref{le:filtration.conditions.lambda.equals.mu} and \ref{le:filtration.conditions.lambda.less.than.mu} give necessary and sufficient conditions (depending on the value of $\mu$) for a filtration $F_W$ to be the fixed part of a module $U \subseteq J(K)$ with $U \simeq M$ and $\lambda(U) = \mu$.  In each case, we need both $F_W \simeq F^{\mbox{\tiny{len}}}_M$ and $F_W \subseteq F^{\mbox{\tiny{len}}}_{J(K)}$.
%$$W_i \subseteq \left\{\begin{array}{ll} J(K)^0_{\{i\}},&\mbox{ if } i<\lambda\\J(K)_{\{i\}},&\mbox{ if } i \geq \lambda;\end{array}\right.$$  
In the case that $\mu = \lambda(J(K)) = p^{i(E/F)}+1 \neq p^n$ it must also be the case that $F_W \not\subseteq \widehat{F^{\mbox{\tiny{len}}}_{J(K)}}$ (which is equivalent to $W_{p^{i(K/F)}+1} \not\subseteq J(K)_{\{p^{i(K/F)}+1\}}^0)$.  Corollary \ref{cor:computing.length.filtration.dimensions.for.JK} computes the dimensions from the filtration $F^{\mbox{\tiny{len}}}_{J(K)}$, and this gives the stated bounds on $\dim\left(M_{\{\ell\}}\right)$.
\end{proof}

Of course, this result is simply a more general version of Theorem \ref{th:counting.to.split.embedding.problem}.

\begin{proof}[Proof of Theorem \ref{th:counting.to.split.embedding.problem}]
The first part of Theorem \ref{th:counting.to.split.embedding.problem} follows directly from Theorem \ref{th:embedding.problem.criterion}.  For the second part, suppose first that $(F^\times K^{\times p})/K^{\times p} = \iota(J(F))$ is infinite, and let $U = \oplus \langle \alpha_i \rangle \subseteq J(K)$ be given so that $U \simeq A$ and $\lambda(U) = p^n$.  For any $f \in \iota(J(F)) \setminus U$, we can create a new module $U_f := \oplus \langle f\alpha_i \rangle$ which has $U_f \simeq U \simeq A$ and $\lambda(U_f) = p^n$.  Moreover, if $f_2 \not\in \langle f_1,U \rangle$, then $U_{f_1} \neq U_{f_2}$.  Hence there are infinitely many modules in $J(K)$ which correspond to a solution to $\xymatrix{A \rtimes G \ar@{->>}[r]& G}$.

So suppose that $(F^\times K^{\times p})/K^{\times p}$ is finite.  This implies that $J(K)$ is finite as well.  If $p^n > \lambda(J(K)) = p^{i(K/F)}+1$, then we may apply Theorem \ref{th:count.for.split.group}; the formulation we have in Theorem \ref{th:counting.to.split.embedding.problem} is simply a re-expression of the quantity from Theorem \ref{th:count.for.split.group}, as per its subsequent remark.  If instead $p^n=\lambda(J(K))$ then we must be in the case $p=2,n=1$ and $i(K/F)=0$, in which case we apply Corollary \ref{cor:count.for.split.inside.split.case}.
\end{proof}

Note that Theorem \ref{th:embedding.problem.criterion} gives a condition for solvability in terms of the values $\mathfrak{D}_{i}$ for all $1 \leq \ell \leq p^n$, even though the terms $\mathfrak{D}_{p^k+j}$ are equal for all $1 \leq j \leq p^{k+1}-p^k$.  Hence we can give a slightly more general result.  In this result we will adopt the following notation for $t \in \{-\infty,0,1,\cdots,n-1\}$: $$t\dotplus 1 = \left\{\begin{array}{ll}0,&\mbox{ if }t=-\infty\\t+1,&\mbox{ if }t\geq 0\end{array}\right.$$
We will also interpret $p^{-\infty}$ as $0$.  (These are in keeping with notation already used in \cite{MSS1}.)
\begin{corollary}\label{cor:weaker.conditions}
The embedding problem $\xymatrix{M \bullet_\mu G \ar@{->>}[r]& G}$ has a solution over $K/F$ if and only if the following conditions hold:
\begin{enumerate}
\item for all $t \in \{-\infty,0,1,\cdots,n-1\}$, we have $$\dim(M_{\{p^{t}+1\}}) \leq \left\{\begin{array}{ll}
\mathfrak{D}_{p^{t \dotplus 1}} + 1,&\mbox{ if }t=i(K/F), p^{i(K/F)}+1=\mu, \mbox{ and }\\&\mbox{ either }p>2,n>1 \mbox{ or }i(K/F)=-\infty\\[10pt]
\mathfrak{D}_{p^{t \dotplus 1}}, &\mbox{ otherwise}.\end{array}\right.
$$%unless $t = i(K/F)$, $\mu = p^{t}+1$ and either $p>2$, $n>1$ or $i(K/F) \neq -\infty$, in which case $\dim\left(M_{\{p^{i(K/F)}+1\}}\right) \leq \mathfrak{D}_{p^{t \dotplus 1}} + 1$.
%\item $\dim(M_{\{p^k+1\}}) \leq \mathfrak{D}_{p^{k+1}} + \mathds{1}_{k=i(K/F)} \cdot \mathds{1}_{p^{i(K/F)}+1 = \mu}$, and
\item $\mu \geq p^{i(K/F)}+1$.  
\end{enumerate}
\end{corollary}
\begin{proof}
Theorem \ref{th:embedding.problem.criterion} gives us necessary and sufficient conditions for the embedding problem to be solvable.  Since the conditions of Theorem \ref{th:embedding.problem.criterion} imply the conditions of this theorem, one direction of the proof is immediate.  For the other, we will suppose that $M$ and $\mu$ satisfy the above conditions, and we'll show these give the conditions for Theorem \ref{th:embedding.problem.criterion}.  Of course, we only need to verify the first condition of Theorem \ref{th:embedding.problem.criterion}, and so we will prove that $\dim\left(M_{\{\ell\}}\right)$ satisfies the appropriate bound for each $1 \leq \ell \leq p^n$.

First, suppose $\ell=1$.  If $i(K/F) \neq -\infty$ or $\mu \neq 1$, then our hypothesis (with $t=-\infty$) gives $\dim\left(M_{\{1\}}\right) \leq \mathfrak{D}_1$.  Otherwise $i(K/F)=-\infty$ and $\mu=1$, and so our hypothesis gives $\dim\left(M_{\{1\}}\right) \leq \mathfrak{D}_1+1$.  These are precisely the conditions for $\ell=1$ in Theorem \ref{th:embedding.problem.criterion}.

Now consider $1<\ell = p^i+j$ where $1 \leq j \leq p^{i+1}-p^i$, and for the time being suppose further that $i\neq i(K/F)$ or $\mu \neq p^{i(K/F)}+1$.  By definition we have $\mathfrak{D}_{p^{i+1}} = \mathfrak{D}_{p^i+j}$, and so our hypothesis gives $$\dim\left(M_{\{p^i+j\}}\right) \leq \dim\left(M_{\{p^i+1\}}\right) \leq \mathfrak{D}_{p^{i+1}} = \mathfrak{D}_{p^i+j},$$ as required by Theorem \ref{th:embedding.problem.criterion}.  

Now suppose $1 < \ell = p^i+j$ where $1 \leq j \leq p^{i+1}-p^i$, and now assume $i = i(K/F)$ and $\mu = p^{i(K/F)}+1$.  If $j=1$ then we have $\dim\left(M_{\{p^i+1\}}\right) \leq \mathfrak{D}_{p^{i+1}}+1$ as required.  For $j>1$,  we claim that $\dim\left(M_{\{p^i+j\}}\right)<\dim\left(M_{\{p^i+1\}}\right)$.  Once we verify this claim, we will have
$$\dim\left(M_{\{p^i+j\}}\right)<\dim\left(M_{\{p^i+1\}}\right) \leq \mathfrak{D}_{p^{i+1}}+1,$$ and therefore $\dim\left(M_{\{p^i+j\}}\right)\leq \mathfrak{D}_{p^{i+1}}$ as desired.

To verify our claim, observe that if $M \subseteq J(K)$ satisfies $\lambda(M) = \mu = p^i+1$, then there is an element $m \in M$ with $\ell(m) = p^i+1$ and $e(m) \neq 0$.  If $\Psi^{p^i}m = \Psi^{p^i}b$ for some $b \in M^0$, then the term $m-b$ would have $\ell(m-b)<p^i+1$ and $e(m-b) \neq 0$.  This would contradict $\lambda(M) = p^i+1$, and so no such $b$ exists.  In particular, it must be that $\Psi^{p^i}m \not\in M_{\{p^i+2\}} \subseteq M_{\{p^i+1\}}^0$.  Hence $$\dim(M_{\{p^i+j\}}) \leq \dim\left(M_{\{p^i+2\}}\right) < \dim(M_{\{p^i+1\}}),$$ and the result follows.
\end{proof}

The previous corollary will be the foundation for a very general automatic realization result.  Before stating the theorem, recall the definition of $\lceil M \rceil$ of an $\F_p[G]$-module $M$ from (\ref{eq:roundup.decomposition.for.general.module}).  For the group $M \bullet_\mu G$ with $1 \leq \mu \leq p^n$ let $\{\alpha_i\}$ be the generators for $M$, and when $\lambda(M)<p^n$ assume that $\alpha_1$ is a generator of $A$ with $\ell(\alpha_1) = \mu$ and $e(\alpha_1) \neq 0$.  Recall that for $1 \leq \ell \leq p^n$, the quantity $\varepsilon(\ell)$ satisfies $p^{\varepsilon(\ell)-1} < \ell \leq p^{\varepsilon(\ell)}$.  Notice that when $\ell=1$ we have $\varepsilon(\ell)=0$; in this case, we will define $\varepsilon(\ell)-1$ to be $-\infty$, which fits with our early conventions involving $-\infty$.  We define $\lfloor A \rceil$ to be the $\F_p[G]$-module with generators $\{\beta_i\}$ subject to the conditions 
$$\ell(\beta_i) = \left\{\begin{array}{ll} p^{\varepsilon(\ell(\alpha_i))},&\mbox{ if }i \neq 1\\p^{\varepsilon(\mu)- 1}+1,&\mbox{ if }i = 1.\end{array}\right.$$ (Note: if $\mu = 1$, our definition of $\varepsilon(1)-1=-\infty$ means $p^{\varepsilon(\mu)-1}+1=1$.)  Again, the notation $\lfloor M \rceil$ is chosen suggestively, since most summands of $M$ have their dimension ``rounded up" to the nearest power of $p$, whereas the summand with nontrivial index has its dimension ``rounded down" to the nearest number which is one more than a power of $p$.  It is also worth noting that $\lceil M \rceil$ and $\lfloor M \rceil$ differ (as $\F_p[G]$-modules) by at most one summand, and so typically the module $\lfloor M \rceil$ is much larger than $M$.

\begin{theorem}\label{th:general.auto.realization}
The group $M \rtimes G$ automatically realizes the group $\lceil M \rceil \rtimes G$.  Furthermore, if $1 \leq \mu < p^n$, define $\tilde \mu=p^{\varepsilon(\mu)-1}+1$. Then $M \bullet_\mu G$ automatically realizes $\left\lfloor M \right\rceil \bullet_{\tilde \mu} G$.
\end{theorem}

\begin{proof}
Suppose first that $F \in \mathfrak{F}(M \rtimes G)$; let $K$ be the fixed field of the subgroup $M \rtimes 0$.  In order to show that there is a solution to the embedding problem $\lceil M \rceil \rtimes G$ over $K/F$, we will appeal to Corollary \ref{cor:weaker.conditions}.  Observe first that if $p=2$ and $n=1$, then $M= \lceil M \rceil$, and so the result is trivial.  Otherwise we must be in the case where $p^{i(K/F)}+1<p^n$, and hence we must show that for all $t \in \{-\infty,0,1,\cdots,n-1\}$ we have $\dim(\lceil M\rceil_{\{p^t+1\}})\leq \mathfrak{D}_{p^{t\dotplus 1}}$.  By construction we have $\dim\left(\lceil M \rceil_{\{p^t+1\}}\right) = \dim\left(M_{\{p^t+1\}}\right)$, and since $\xymatrix{M \rtimes G \ar@{->>}[r] &G}$ is solvable over $K/F$, the previous Corollary tells us $\dim\left(M_{\{p^t+1\}}\right) \leq \mathfrak{D}_{p^{t \dotplus 1}}$.  This completes the proof of the first statement.

Now suppose that $F \in \mathfrak{F}(M \bullet_\mu G)$; as before, let $K$ be the fixed field of the subgroup generated by $M$.  By Corollary \ref{cor:weaker.conditions}, %we know that
%\begin{enumerate}
%\item $\Delta(A^G) \leq \mathfrak{D}_{\{1\}} + \mathds{1}_{i(K/F) = -\infty}\cdot \mathds{1}_{1=\mu}$
%\item $\Delta(A_{\{p^k+1\}}) \leq \mathfrak{D}_{\{p^{k+1}\}} + \mathds{1}_{k=i(K/F)} \cdot \mathds{1}_{p^{i(K/F)}+1 = %\mu}$, and
%\item $\lambda \geq p^{i(K/F)}+1$.  
%\end{enumerate}  
we need to show that
\begin{enumerate}
\item for all $t \in \{-\infty,0,1,\cdots,n-1\}$, we have $$\dim(\lfloor M \rceil_{\{p^{t}+1\}}) \leq \left\{\begin{array}{ll}
\mathfrak{D}_{p^{t \dotplus 1}} + 1,&\mbox{ if }t=i(K/F), p^{i(K/F)}+1=\tilde \mu, \mbox{ and }\\&\mbox{ either }p>2,n>1 \mbox{ or }i(K/F)=-\infty\\
\mathfrak{D}_{p^{t \dotplus 1}}, &\mbox{ otherwise}.\end{array}\right.
$$%unless $t = i(K/F)$, $\mu = p^{t}+1$ and either $p>2$, $n>1$ or $i(K/F) \neq -\infty$, in which case $\dim\left(M_{\{p^{i(K/F)}+1\}}\right) \leq \mathfrak{D}_{p^{t \dotplus 1}} + 1$.
%\item $\dim(M_{\{p^k+1\}}) \leq \mathfrak{D}_{p^{k+1}} + \mathds{1}_{k=i(K/F)} \cdot \mathds{1}_{p^{i(K/F)}+1 = \mu}$, and
\item $\tilde \mu \geq p^{i(K/F)}+1$.  
\end{enumerate}

%Now since $\xymatrix{M \bullet_\mu G\ar@{->>}[r]&G}$ is solvable over $K/F$, we know that $\Delta(A^G) \leq \mathfrak{D}_1 + \mathds{1}_{i(K/F) = -\infty}\cdot \mathds{1}_{1=\lambda}$; since $\Delta(A^G) = \Delta(\lfloor A \rceil^G) = \rk(A)$ and $\mathds{1}_{1 = \lambda} = \mathds{1}_{1 = p^{\lfloor \log_p(\lambda-1)\rceil}+1}$, the first property holds.  

For the first condition, we have $\dim(\lfloor M \rceil_{\{p^t+1\}}) = \dim(M_{\{p^{t}+1\}})$, and since $\xymatrix{M \bullet_\mu G\ar@{->>}[r]&G}$ is solvable over $K/F$ we have $$\Delta(M_{\{p^t+1\}}) \leq \left\{\begin{array}{ll}
\mathfrak{D}_{p^{t \dotplus 1}} + 1,&\mbox{ if }t=i(K/F), p^{i(K/F)}+1=\mu, \mbox{ and }\\&\mbox{ either }p>2,n>1 \mbox{ or }i(K/F)=-\infty\\
\mathfrak{D}_{p^{t \dotplus 1}}, &\mbox{ otherwise}.\end{array}\right.$$
If $\mu \neq p^{i(K/F)}+1$ then we have $\dim\left(M_{\{p^t+1\}}\right) \leq \mathfrak{D}_{p^{t \dotplus 1}}$ in all cases, and otherwise we have $\mu = \tilde \mu$, in which case we recover precisely the desired conditions.

Now we check $\tilde \mu \geq p^{i(K/F)}+1$.  Suppose to the contrary that $\tilde \mu = p^{\varepsilon(\mu)-1}+1<p^{i(K/F)}+1$.  This implies $i(K/F) >\varepsilon(\mu)-1$.  On the other hand, the solvability of $\xymatrix{M \bullet_\mu G \ar@{->>}[r]&G}$ and the definition of $\varepsilon(\mu)$ tell us $p^{i(K/F)}+1 \leq \mu \leq p^{\varepsilon(\mu)}$, from which we derive $i(K/F)<\varepsilon(\mu)$.  Hence $\varepsilon(\mu)-1<i(K/F)<\varepsilon(\mu),$ a clear contradiction.
\end{proof}

We wish to emphasize that the results we have presented here are certainly not the only conclusions one can draw, but are chosen simply to give an indication of the wide-sweeping results one can obtain from the parameterization from Theorem \ref{th:main.theorem} together with the computed module structures in Propositions \ref{prop:module.structure.char.not.p} and \ref{prop:module.structure.char.p}.  For a more detailed look at how these results are related to automatic realizations and realization multiplicities for the groups $H_{p^3}$ and $M_{p^3}$ from Example \ref{ex:hp3.and.mp3}, as well as related groups, see the author's forthcoming collaboration \cite{CMS}.


\begin{thebibliography}{99}

\bibitem{BS} {\sc J.~Berg}, {\sc A.~Schultz}. $p$-groups have unbounded realization multiplicity. \emph{Proc.~Amer.~Math.~Soc.}, electronically published on March 11, 2014, DOI: http://dx.doi.org/10.1090/S0002-9939-2014-11967-8 (to appear in print).

\bibitem{Bo} {\sc Z.~I.~Borevi\v{c}}. The multiplicative group of cyclic $p$-extensions of a local field. \emph{Trudy Mat.~Inst.~Steklov} {\bf 80} (1965) 16--29.  English translation, \emph{Proc.~Steklov Inst.~Math.~No.~80 (1965): Algebraic number theory and representations} (ed {\sc D.~K.~Fadeev}) (American Mathematical Society, Providence, RI, 1968), 15--30.

\bibitem{Br} {\sc G.~Brattstr\"{o}m}. On p-groups as Galois groups. \textit{Math.~Scandinavica} {\bf 65} (1989), no.~2, 165--174.

\bibitem{CMS} {\sc S.~Chebolu}, {\sc J.~Min\'{a}\v{c}}, {\sc A.~Schultz}.  Non-abelian Galois groups of order $p^3$, Galois Modules and the norm residue homomorphism.  \emph{Manuscript}.

\bibitem{DM} {\sc P.~Damey}, {\sc J.~Martinet}. Plongement d'une extension quadratique dans une extension quaternionienne. \emph{J.~Reine Angew.~Math.}{\bf 262/263} (1973), 323--338.

\bibitem{DP} {\sc P.~Damey}, {\sc J.-J.~Payan}. Existence et construction des extensions Galoisiennes et non-ab'{e}liennes de degr\'{e} $8$ d'un corps de caract\'{e}ristique diff\'{e}rente de $2$. \emph{J.~Reine Angew.~Math.} {\bf 244} (1970), 37--54.

\bibitem{D} {\sc R.~Dedekind}. Konstruktion von Quaternionk\"{o}rpern. Gasammelte Mathematische Werke, Band 2, Vieweg, Braunschweig, 1931, 376--384.

\bibitem{F} {\sc D.~K.~Fadeev}. On the structure of the reduced multiplicative group of a cyclic extension of a local field. \emph{Izv.~Akad.~Nauk SSSR Ser.~Math.} {\bf 24} (1960), 145--152.

\bibitem{GS1} {\sc H.~Grundman}, {\sc T.~Smith}. Automatic realizability of Galois groups of order $16$. \emph{Proc.~Amer.~Math.~Soc.} {\bf 124} (1996), 2631--2640.

\bibitem{GS2} {\sc H.~Grundman}, {\sc T.~Smith}. Realizability and automatic realizability of Galois groups of order $32$. \emph{Cent.~Eur.~J.~Math.} {\bf 8} (2010), no.~2, 244-260.

\bibitem{GS3} {\sc H.~Grundman}, {\sc T.~Smith}. Galois realizability of groups of order $64$. \emph{Cent.~Eur.~J.~Math.} {\bf 8} (2010), no.~5, 846--854.

\bibitem{GSS} {\sc H.~Grundman}, {\sc T.~Smith}, {\sc J.~Swallow}. Groups of order $16$ as Galois groups. \emph{Expo.~Math.} {\bf 13} (1995), 289--319.

\bibitem{J1} {\sc C.~U.~Jensen}, On the representations of a group as a Galois group over an arbitrary field. Th\'eorie des nombres (Quebec, PQ, 1987), 441--458. Berlin: de Gruyter, 1989.

\bibitem{J2} {\sc C.~U.~Jensen}. Finite groups as Galois groups over arbitrary fields.  Proceedings of the International Conference on Algebra, Part 2 (Novosibirsk, 1989), 435--448. Contemp.~Math.  {\bf 131}, Part 2. Providence, RI: American Mathematical Society, 1992.

\bibitem{J3}  {\sc C.~U.~Jensen}. Elementary questions in Galois theory. Advances in algebra and model theory (Essen, 1994; Dresden, 1995), 11--24.  Algebra Logic Appl. {\bf 9}. Amsterdam: Gordon and Breach, 1997.

\bibitem{JLY} {\sc C.~U.~Jensen}, {\sc A.~Ledet}, {\sc N.~Yui}. Generic polynomials: constructive aspects of the inverse Galois problem.  Mathematical Sciences Research Institute Publications 45. Cambridge: Cambridge University Press, 2002.

\bibitem{JP1} {\sc C.~U.~Jensen},{\sc A.~Prestel}. Unique realizability of finite abelian $2$-groups as Galois groups. \textit{J. Number Theory} {\bf 40} (1992), no.~1, 12--31.

\bibitem{JP2} {\sc C.~U.~Jensen},{\sc A.~Prestel}. How often can a finite group be realized as a Galois group over a field? \textit{Manuscripta Math.} {\bf 99} (1999), 223--247.

\bibitem{L} {\sc A.~Ledet}. On $2$-groups as Galois groups. \emph{Canad.~J.~Math.} {\bf 47} (1995), no.~6, 1253--1273.

\bibitem{LMSSembed}  {\sc N.~Lemire}, {\sc J.~Min\'a\v{c}}, {\sc A.~Schultz}, {\sc J.~Swallow}. Galois module structure of Galois cohomology for embeddable cyclic extensions of degree $p^n$. \emph{J.~London Math.~Soc.} {\bf 81} (2010), no.~3, 525--543.

\bibitem{MN} {\sc R.~Massy}, {\sc T.~Nguyen-Quang-Do}. Plongement d'une extension de degr\'{e} $p^2$ dans une surextension non ab\'{e}lienne de degr\'{e} $p^3$: \'{e}tude locale-globale. \emph{J.~Reine Angew.~Math.} {\bf 291}(1977), 149--161.

\bibitem{M1} {\sc I.~Michailov}. Four non-abelian groups of order $p^4$ as Galois groups. \emph{J.~Alg.} {\bf 307} (2007), 287--299.

\bibitem{M2} {\sc I.~Michailov}. Groups of order $32$ as Galois groups. \emph{Serdica Math.~J.} {\bf 33} (2007), no.~1, 1--34.

\bibitem{M3} {\sc I.~Michailov}. Induced orthogonal representations of Galois groups. \emph{J.~Alg.} {\bf 322} (2009), 3713--3732.

\bibitem{M4} {\sc I.~Michailov}.  On Galois cohomology and realizability of $2$-groups as Galois groups. \emph{Cent.~Eur.~J.~Math.} {\bf 9} (2011), no.~2, 403--419.

\bibitem{M5} {\sc I.~Michailov}.  On Galois cohomology and realizability of $2$-groups as Galois groups II. \emph{Cent.~Eur.~J.~Math.} {\bf 9} (2011), no.~6, 1333--1343.

\bibitem{MS1} {\sc J.~Min\'a\v{c}}, {\sc J.~Swallow}, Galois module structure of $p$th-power classes of extensions of degree $p$. Israel J.~Math. {\bf 138} (2003), 29--42.

\bibitem{MS2} {\sc J.~Min\'a\v{c}}, {\sc J.~Swallow},  Galois embedding
problems with cyclic quotient of order $p$. Israel J.~Math.
{\bf 145} (2005), 93--112.

\bibitem{MSS1} {\sc J.~Min\'a\v{c}}, {\sc A.~Schultz}, {\sc J.~Swallow}, Galois module structure of the $p$th-power classes of cyclic extensions of degree $p^n$.  Proc.~London Math.~Soc. {\bf 92} (2006), no.~2, 307--341.

\bibitem{MSSauto} {\sc J.~Min\'a\v{c}}, {\sc A.~Schultz}, {\sc J.~Swallow}. Automatic realizations of Galois groups with cyclic quotient of order $p^n$. \emph{J.~Th\'{e}or.~Nombres Bordeaux} {\bf 20} (2008), 419--430.

\bibitem{Sh} {\sc V.~Shirbisheh}. Galois embedding problems with abelian kernels of exponent $p$. VDM Verlag, 2009.

%\bibitem{Sh2} {\sc V.~Shirbisheh}. Module homomorphisms of group algebras of cyclic $p$-groups in characteristic $p$. \emph{Int.~Electronic J.~Alg.} {\bf 7} (2010), 59--77.

\bibitem{Sw} {\sc J.~Swallow}. Central $p$-extensions of $(p,p,\cdots,p)$-type Galois groups. \emph{J.~Alg.} {\bf 186}(1996), 277--298.

\bibitem{Wa} {\sc W.~Waterhouse},  The normal closures of certain Kummer extensions. Canad. Math.~Bull. {\bf 37} (1994), no.~1,
133--139.

\bibitem{Wh} {\sc G.~Whaples}, Algebraic extensions of arbitrary fields. \emph{Duke Math.~J.} {\bf 24} (1957), 201--204.

\bibitem{Wi} {\sc E.~Witt}. Konstruktion von galoisschen K\"{o}rpern der Charakteristik $p$ zu vorgegebener Gruppe der Ordnung $p^f$. \emph{J.~Reine Angew.~Math.} {\bf 174} (1936), 237--245.
\end{thebibliography}
\end{document}